\documentclass[10pt,a4paper,reqno]{amsart}
\usepackage{amsthm}
\usepackage{amsmath}
\usepackage{amssymb}
\usepackage{mathtools}
\mathtoolsset{showonlyrefs=true}

\usepackage[font=small]{caption}

\usepackage[shortlabels]{enumitem}
\usepackage{xcolor,graphicx}
\usepackage{multirow}

\makeatletter

\theoremstyle{plain}
\newtheorem{thm}{Theorem}
  \theoremstyle{definition}
  
  \theoremstyle{remark}
  \newtheorem{rem}[thm]{Remark}
  \theoremstyle{plain}
  \newtheorem{prop}[thm]{Proposition}
  \theoremstyle{plain}
  \newtheorem{lem}[thm]{Lemma}
  \theoremstyle{plain}
  \newtheorem{cor}[thm]{Corollary}
 \theoremstyle{definition}
  
  \theoremstyle{remark}
  \newtheorem*{rem*}{Remark}

  \theoremstyle{definition}

\usepackage{mathrsfs}

\newtheorem*{question*}{\it{QUESTION}}

\theoremstyle{plain}

\newcommand{\N}{\mathbb{N}}
\newcommand{\R}{{\mathbb{R}}}
\newcommand{\C}{{\mathbb{C}}}
\newcommand{\Z}{{\mathbb{Z}}}
\newcommand{\dd}{{{\rm d}}}
\newcommand{\ii}{{\rm i}}

\newcommand{\spn}{\mathop\mathrm{span}\nolimits}
\newcommand{\Dom}{\mathop\mathrm{Dom}\nolimits}

\newcommand{\Ker}{\mathop\mathrm{Ker}\nolimits}
\newcommand{\spec}{\mathop\mathrm{spec}\nolimits}

\renewcommand{\Re}{\mathop\mathrm{Re}\nolimits}
\renewcommand{\Im}{\mathop\mathrm{Im}\nolimits}

\newcommand{\Res}{\mathop\mathrm{Res}\nolimits}

\newcommand{\sn}{\mathop\mathrm{sn}\nolimits}
\newcommand{\cn}{\mathop\mathrm{cn}\nolimits}
\newcommand{\cs}{\mathop\mathrm{cs}\nolimits}

\makeatother

\begin{document}

\title[Spectral analysis of two doubly infinite Jacobi matrices]{Spectral analysis of two doubly infinite Jacobi matrices with exponential entries}

\author{Mourad E. H. Ismail}
\address[Mourad E. H. Ismail]{
Department of Mathematics, University of Central Florida, Orlando, Florida 32816, USA
 \& Department of Mathematics, King Saud University, Riyadh, Saudi Arabia}
\email{mourad.eh.ismail@gmail.com}
 
\author{Franti\v sek \v Stampach}
\address[Franti\v sek \v Stampach]{
	Department of Mathematics, 
	Stockholm University,
	Kr\"{a}ftriket 5,
	SE - 106 91 Stockholm, Sweden}
\email{stampach@math.su.se}



\date{\today}

\begin{abstract}
We provide a complete spectral analysis of all self-adjoint operators acting on $\ell^{2}(\Z)$ which are associated with two doubly infinite Jacobi matrices with entries given by
\[
 q^{-n+1}\delta_{m,n-1}+q^{-n}\delta_{m,n+1}
\]
and
\[
 \delta_{m,n-1}+\alpha q^{-n}\delta_{m,n}+\delta_{m,n+1},
\]
respectively, where $q\in(0,1)$ and $\alpha\in\R$. As an application, we derive orthogonality relations for the Ramanujan entire function and the third Jackson $q$-Bessel function.
\end{abstract}

\maketitle

\section{Introduction}

This paper is devoted to an investigation of spectral properties of two families of unbounded Jacobi operators
acting on $\ell^{2}(\Z)$. They are determined by two doubly infinite tridiagonal matrices whose entries
are given in terms of powers of a parameter $q$ where we always assume $0<q<1$.

First, we consider an operator $A=A(q)$ acting on the $n$-th vector of the standard basis of~$\ell^{2}(\Z)$ as
\begin{equation}
 Ae_{n}=q^{-n+1}e_{n-1}+q^{-n}e_{n+1}, \quad n\in\Z.
\label{eq:Ae_n}
\end{equation}
The operator $A$ endowed with the domain $\Dom A=\spn\{e_{n}\mid n\in\Z\}$ is a symmetric operator
with deficiency indices $(1,1)$. Thus, there is a whole one-parameter family of self-adjoint extensions 
of $A$ whose domains as well as their fundamental spectral properties will be described in details.

The recurrence relation determined by \eqref{eq:Ae_n} has a long history. It appeared earlier in the proofs of the Rogers--Ramanujan 
identities and the Ramanujan continued fraction. For example, Schur used them in his analysis of the identities, see \cite{Schur}. They also appeared in the 
generalization of the Rogers--Ramanujan identities, referred to as the $m$-version, given by Garrett, Ismail, and Stanton in 
\cite{GarrretIsmailStanton}.  Carlitz also used \eqref{eq:Ae_n} to study an $m$-version of the Rogers--Ramanujan 
identities  for $m \le 0$ in   \cite{Carlitz}.  For details, we refer the interested reader to \cite[Sec.~13.5 and 13.6]{Ismail}.

The semi-infinite parts of $A$, i.e., the semi-infinite Jacobi operators determined by \eqref{eq:Ae_n} with indices being restricted to $\Z_{\pm}$ (non-negative/non-positive integers)
have been studied before and their spectral properties are well known, see \cite{AlSalamIsmail, ChenIsmail, Stampach}. The results from \cite{AlSalamIsmail, ChenIsmail, Stampach}
are not formulated as spectral properties of the corresponding Jacobi operators but equivalently as specific properties of corresponding families of orthogonal polynomials. 
It is worth mentioning that the spectral properties of these semi-infinite Jacobi operators were always accompanied with interesting mathematical objects: the Ramanujan continued fraction,
the Ramanujan entire function and its zeros \cite{AlSalamIsmail}; and interesting mathematical phenomena like indeterminate Hamburger moment problem with known Nevanlinna parametrization 
functions~\cite{ChenIsmail, Stampach}.  

The second part of the paper is devoted to a study of a second operator $B=B(\alpha,q)$ whose action is defined by the formula
\begin{equation}
 Be_{n}=e_{n-1}+\alpha q^{-n}e_{n}+e_{n+1}, \quad n\in\Z,
\label{eq:Be_n}
\end{equation}
where $\alpha\in\R$. The relation \eqref{eq:Be_n} gives rise to a unique self-adjoint discrete Schr{\" o}dinger operator
with a potential depending on a coupling constant $\alpha$. Again, the spectral properties of the corresponding semi-infinite 
Jacobi operators has been studied before. For some results within the theory of orthogonal polynomials, see \cite{IsmailMulla}.

Throughout this article, we shall follow the standard notation for basic hypergeometric series as in \cite{AndrewsAskeyRoy, GasperRahman}. For the general theory on spectral analysis
of Jacobi operators, the reader is referred to \cite{Akhiezer, Berezanskii, Teschl}.

The paper is
organized as follows. In Section~\ref{sec:opA}, the operator $A$ is treated. It is shown the corresponding difference expression is in the limit circle case at $+\infty$
and domains of all self-adjoint extensions $A_{t}$, $t\in\R\cup\{\infty\}$, are described by respective boundary conditions. We show the continues spectrum equals
essential spectrum and they coincide with $\{0\}$ for all the self-adjoint extensions. A suitable reparametrization of the parameter labeling the self-adjoint extensions
allows us to obtain the discrete part of the spectrum of $A_{t}$ in a full explicit form. Here the crucial role is played by Jacobian elliptic functions and elliptic integrals. 
Further, we derive a formula for corresponding eigenvectors in terms of basic hypergeometric series and compute their $\ell^{2}$-norm. As an application of deduced spectral 
properties of $A$, we obtain several orthogonality relations for the Ramanujan entire function.

In Section~\ref{sec:opB}, the spectral analysis of $B$ is treated. We show the essential spectrum of $B$ coincides with the absolutely continuous spectrum of $B$ and
equals the spectrum of the free discrete Schr{\" o}dinger operator $[-2,2]$. In addition, the discrete spectrum of $B$ is found fully explicitly, corresponding eigenvectors are expressed in terms of 
basic hypergeometric series and their $\ell^{2}$-norm is computed. Moreover, we derive a formula for an arbitrary matrix element of the spectral measure of $B$ by finding its
density in the absolutely continues part of the spectrum. At last, some consequences concerning third Jackson's $q$-Bessel functions are discussed.

\section{Spectral analysis of operator $A$}
\label{sec:opA}

\subsection{General properties of operator $A$}

Let us denote $\Z_{+}=\{0,1,2,\dots\}$, $\Z_{-}=-\Z_{+}$ and $P_{\pm}$ the orthogonal projections on $\spn\{e_{n} \mid n\in\Z_{\pm}\}$.
Then $P_{+}AP_{+}=0\oplus A_{+}$ and $P_{-}AP_{-}=A_{-}\oplus 0$ where $A_{+}$ act on $\ell^{2}(\Z_{+})$ and $A_{-}$ acts on $\ell^{2}(\Z_{-})$, respectively. 
Clearly, 
\begin{equation}
 A=A_{-}\oplus 0 + 0\oplus A_{+},
\label{eq:A_decomp}
\end{equation}
since $\langle e_{0},Ae_{0}\rangle=0$.

\begin{prop}
 Deficiency indices of $A$ are $(1,1)$.
\end{prop}

\begin{proof}
 By the theorem of Masson and Repka, see \cite{MassonRepka}, deficiency indices of $A$ equals the sum of
 deficiency indices of $A_{-}$ and $A_{+}$. Deficiency indices of $A_{-}$ are $(0,0)$ for it is a bounded operator,
 while deficiency indices of $A_{+}$ are known to be $(1,1)$, see \cite{ChenIsmail, Stampach}. 
\end{proof}

\begin{prop}\label{prop:spec_A_c_eq_0}
 Let $A_{t}$ be a self-adjoint extension of $A$, then
 \[
 \spec_{ess}(A_t)=\spec_{c}(A_t)=\{0\}.
 \]
\end{prop}

\begin{proof}
 Since the deficiency indices of $A_{+}$ are $(1,1)$, any self-adjoint extension $(A_{+})_{t}$ of $A_{+}$ is an operator with discrete spectrum, 
 see, for example, \cite[Lem.~2.19]{Teschl}. Hence the essential spectrum of the full operator $0\oplus(A_{+})_{t}$ (acting on $\ell^{2}(\Z)$)
 coincides with $\{0\}$, for $0$ is an eigenvalue of infinite multiplicity. The remainder $A_{-}\oplus0$,
 in the decomposition
 \[
  A_{t}=A_{-}\oplus 0 + 0\oplus(A_{+})_{t},
 \]
 is a compact operator. Consequently, the equality $\spec_{ess}(A_t)=\{0\}$ follows from 
 the Weyl theorem on the stability of the essential spectrum of a self-adjoint operator.
 
 It is easy to check that the sequences $p=\{p_{n}\}_{n\in\Z}$ and $q=\{q_{n}\}_{n\in\Z}$, given by equalities
 \begin{equation}
  p_{2n-1}=q_{2n}=0, \quad p_{2n}=q_{2n+1}=(-q)^{n},
 \label{eq:p_q_def}  
 \end{equation}
 for $n\in\Z$, are two linearly independent solutions of the eigenvalue equation $A\psi=0$. Clearly, neither the solution $p$ nor $q$ is square summable
 at $-\infty$. Thus, $0$ can not be an eigenvalue of $A_t$, for any $t$, since the corresponding eigenvector
 would have to be a nontrivial linear combination of $p$ and $q$.
\end{proof}

\begin{rem}
 Recall the essential spectrum of a self-adjoint operator contains either eigenvalues of infinite algebraic multiplicity, or
 accumulation points of the spectrum. Taking into account the previous proposition and the fact that the multiplicity of an 
 eigenvalue of a Jacobi operator can not exceed two, $0$ has to be the only (finite) accumulation point of eigenvalues for any~$A_{t}$.
\end{rem}

\subsection{Solutions of the eigenvalue equation}

First, we provide two general solutions of the eigenvalue equation
\begin{equation}
 q^{-n+1}\psi_{n-1}+q^{-n}\psi_{n+1}=x\psi_{n}, \quad \forall n\in\Z.
\label{eq:eigenval_diff_eq_A}
\end{equation}

\begin{prop}\label{prop:sol_qExp}
 For any $x\in\C$, the sequences $\psi^{\pm}$, where 
 \begin{equation}
  \psi_{n}^{\pm}=\psi_{n}^{\pm}(x):=(\pm\ii)^{n}q^{n/2}{}_{1}\phi_{1}\!\left(0;-q;q,\mp\ii x q^{n+1/2}\right)\!, \quad n\in\Z,
 \label{eq:psipm_def}
 \end{equation}
 form a couple of linearly independent solutions of the difference equation \eqref{eq:eigenval_diff_eq_A}.
\end{prop}

\begin{proof}
 The verification of this statement has been essentially given within the proof of Proposition~4 in \cite{Stampach}.
 Is is straightforward to check that $\psi^{\pm}$ solve the equation \eqref{eq:eigenval_diff_eq_A}. Next, for the Wronskian 
 \[
 W(\psi^{+},\psi^{-})=q^{-n}\left(\psi_{n+1}^{+}\psi_{n}^{-}-\psi_{n}^{+}\psi_{n+1}^{-}\right)\!,
 \]
 one has $W(\psi^{+},\psi^{-})=2\ii q^{1/2}\neq0$, as one easily deduces by taking limit $n\to\infty$ in the above formula.
 This implies the linear independence of $\psi^{+},\psi^{-}$.
\end{proof}

\begin{rem}
 Note that the solutions $\psi^{\pm}$ can be expressed in terms of a $q$-analogue to the exponential function. Namely,
 \[
  \psi_{n}^{\pm}(x)=(\pm\ii)^{n}q^{n/2}\exp_{q^{2}}\!\left(\pm\frac{\ii x q^{n+1/2}}{1-q^{2}}\right)\!, \quad n\in\Z
 \]
 where 
 \[
  \exp_{q^{2}}(z)={}_{1}\phi_{1}\!\left(0;-q^{1/2};q^{1/2},-(1-q)z\right)\!,
 \]
 see \cite[Eq.~(1.3.26)]{GasperRahman}. The connection between solutions of \eqref{eq:eigenval_diff_eq_A} and $q$-exponential functions has already been 
 observed in \cite{ChenIsmail} and play an important role in the study of some exceptional solutions of the corresponding indeterminate Hamburger moment 
 problem, see \cite{Stampach} for details.
\end{rem}

It is clear that both solutions $\psi^{\pm}$ are square summable at $+\infty$. However, 
it can be shown that 
this is not the case at $-\infty$. Hence, one may expect there are non-trivial coefficients $a=a(x)$ and 
$b=b(x)$ such that the linear combination $a\psi^{+}+b\psi^{-}\in\ell^{2}(\Z_{-})$, at least for some 
$x\in\R$. In fact, this is true for all $x\in\C\setminus\{0\}$, as it is shown in the Proposition~\ref{prop:l2_sol} below.

Let us denote
\begin{equation}
\theta_{q}(x):=\left(x,qx^{-1};q\right)_{\infty} = \frac{1}{(q;q)_\infty} \sum_{n=-\infty}^\infty  q^{\binom{n}{2}}(-x)^n, \quad x\in\C,
\label{eqdeftheta}
 \end{equation}
where the second equality is known as the Jacobi triple product identity, see \cite[Eq.~(II.28)]{GasperRahman}. 
By using \eqref{eqdeftheta}, it is easy to verify that 
for $x\neq 0$, 
 \begin{equation}
  (-1;q)_{\infty}\,_{0}\phi_{1}\left(-;0;q^{2},q^{5}x^{-2}\right)\!
  =\theta_{q}\!\left(-q^{-1}x\right){}_{1}\phi_{1}\left(0;-q;q,x\right)+\theta_{q}\!\left(q^{-1}x\right)
  {}_{1}\phi_{1}\left(0;-q;q,-x\right)\!,
 \label{eq:3trr}
  \end{equation}
a connection formula needed in the proof of the following proposition.

\begin{prop}\label{prop:l2_sol}
 For all $x\in\C\setminus\{0\}$, the sequence
 \begin{equation}
  \varphi(x):=\theta_{q}\left(-\ii q^{-1/2}x\right)\psi^{(-)}(x)+\theta_{q}\left(\ii q^{-1/2}x\right)\psi^{(+)}(x),
  \label{eq:phi_def}
 \end{equation}
 where $\theta_{q}$ is given by \eqref{eqdeftheta}, is the non-trivial solution of \eqref{eq:eigenval_diff_eq_A}
 which belongs to $\ell^{2}(\Z)$. Moreover, within the space $\ell^{2}(\Z)$, this solution is given uniquely up to 
 a multiplicative constant.
\end{prop}

\begin{proof}
 First note that functions $\theta_{q}\left(\pm x\right)$ have no common zeros. Consequently, $\varphi(x)\neq0$,
 for all $x\in\C\setminus\{0\}$, due the linear independence of $\psi^{(+)}(x),\psi^{(-)}(x)$.
 
 Clearly, $\varphi(x)$ solves \eqref{eq:eigenval_diff_eq_A}, for, by Proposition \ref{prop:sol_qExp}, it is a 
 linear combination of solutions of~\eqref{eq:eigenval_diff_eq_A}.
 
 Next, to show that $\varphi(x)\in\ell^{2}(\Z)$, it suffices to verify the square summability of $\varphi(x)$ 
 at~$-\infty$. By writing $\ii q^{n+1/2}x$ instead of $x$ in the connection formula \eqref{eq:3trr}, using the identity
 \begin{equation}
  \theta_{q}\left(q^{k}x\right)=(-1)^{k}x^{-k}q^{-k(k-1)/2}\theta_{q}\left(x\right), \quad k\in\Z,
 \label{eq:theta_power_q_id}
 \end{equation}
 together with definitions \eqref{eq:psipm_def} and \eqref{eq:phi_def}, one arrives at the formula
 \begin{equation}
  \varphi_{n}(x)=(-1;q)_{\infty}\,x^{n}q^{n(n-1)/2}\,_{0}\phi_{1}\left(-;0;q^{2},-q^{-2n+4}x^{-2}\right)\!, \quad n\in\Z,\ x\neq0.
 \label{eq:varphi_2nd_expr}
 \end{equation}
 The expression on the RHS is readily seen to be square summable for $n$ at $-\infty$ for all $x\neq0$.
 
 Finally, to prove the uniqueness, take another solution $\chi$ of \eqref{eq:eigenval_diff_eq_A} such that $\chi,\varphi$
 are linearly independent. The Wronskian $q^{-n}\left(\varphi_{n+1}\chi_{n}-\varphi_{n}\chi_{n+1}\right)$ is thus a nonzero
 constant. From the expression for the Wronskian, it is clear however, the sequence $\chi_{n}$ can not even be bounded as $n\to-\infty$, 
 since otherwise the Wronskian would vanish.
\end{proof}

Next, we derive a formula for the $\ell^{2}$-norm of the solution \eqref{eq:phi_def} for $x\in\R\setminus\{0\}$. First, however,
we prove a lemma which we will use during the derivation. For this purpose, we define the function
\begin{equation}
 \xi_{q}(z):=\sum_{n=-\infty}^{\infty}\frac{q^{n/2}}{1+zq^{n-1/2}},
\label{eq:def_xi} 
\end{equation}
for $z\in\C\setminus\{-q^{n+1/2}\mid n\in\Z\}$. Clearly, the function $\xi_{q}$ is meromorphic with simple poles located at the points
$-q^{n+1/2}$, $n\in\Z$.

\begin{lem}
 For all $z\notin\{-q^{n+1/2}\mid n\in\Z\}\cup\{0\}$, it holds
 \begin{equation}
  \xi_{q}(z)=\frac{\left(q;q\right)_{\infty}^{2}}{\left(q^{1/2};q\right)_{\infty}^{\!2}}\frac{\theta_{q}\left(-z\right)}{\theta_{q}\left(-q^{-1/2}z\right)}.
  \label{eq:xi_eq_ratio_theta}
 \end{equation}
\end{lem}

\begin{proof}
 First, we rewrite the Mittag-Leffler expansion \eqref{eq:def_xi} in the Laurent series form. By splitting the sum in \eqref{eq:def_xi} into two sums with the summation indices
 ranging positive and non-negative integers, respectively, and making some simple manipulations one gets
 \begin{equation}
  \xi_{q}(z)=\chi_{q}(z)+z^{-1}\chi_{q}\left(z^{-1}\right)
 \label{eq:xi_eq_sum_chi}
 \end{equation}
 where
 \[
  \chi_{q}(z):=\sum_{n=1}^{\infty}\frac{q^{n/2}}{1+zq^{n-1/2}}.
 \]
 Further, one has
 \begin{align*}
  \chi_{q}(z)&=\sum_{n=1}^{\infty}q^{n/2}\sum_{k=0}^{\infty}\left(-z\right)^{k}q^{(2n-1)k/2}=\sum_{k=0}^{\infty}q^{-k/2}\left(-z\right)^{k}\sum_{n=1}^{\infty}q^{(2k+1)n/2}\\
  &=\sum_{k=0}^{\infty}\frac{q^{(k+1)/2}}{1-q^{k+1/2}}\left(-z\right)^{k}\!,
 \end{align*}
 for $|z|<q^{-1/2}$. By plugging the last formula into \eqref{eq:xi_eq_sum_chi}, one obtains the Laurent series expansion
 \begin{equation}
  \xi_{q}(z)=\sum_{k=-\infty}^{\infty}\frac{q^{(k+1)/2}}{1-q^{k+1/2}}\left(-z\right)^{k}
 \label{eq:xi_laurent_ser}
 \end{equation}
 which converges if $q^{1/2}<|z|<q^{-1/2}$. 
 
 The formula \eqref{eq:xi_laurent_ser} admits to express $\xi_{q}$ as the bilateral basic hypergeometric series ${}_{1}\psi_{1}$, see \cite[Chp.~5]{GasperRahman},
 which can be, in turn, sum up by the Ramanujan bilateral summation formula \cite[Eq.~(5.2.1)]{GasperRahman}. Namely, one has
 \[
  \xi_{q}(z)=\frac{q^{1/2}}{1-q^{1/2}}{}_{1}\psi_{1}\left(q^{1/2};q^{3/2};q,-q^{1/2}z\right)=q^{1/2}\frac{\left(q,q,-qz,-z^{-1};q\right)_{\infty}}{\left(q^{1/2},q^{1/2},-q^{1/2}z,-q^{1/2}z^{-1};q\right)_{\infty}}.
 \]
 The last formula written in terms of theta functions reads
 \[
  \xi_{q}(z)=q^{1/2}\frac{\left(q;q\right)_{\infty}^{2}}{\left(q^{1/2};q\right)_{\infty}^{\!2}}\frac{\theta_{q}\left(-z^{-1}\right)}{\theta_{q}\left(-q^{1/2}z\right)}
  =\frac{\left(q;q\right)_{\infty}^{2}}{\left(q^{1/2};q\right)_{\infty}^{\!2}}\frac{\theta_{q}\left(-z\right)}{\theta_{q}\left(-q^{-1/2}z\right)}
 \]
 where we used \eqref{eq:theta_power_q_id} and the identity
 \begin{equation}
  \theta_{q}\left(z\right)=-z\theta_{q}\left(z^{-1}\right)
 \label{eq:theta_recip_arg}
 \end{equation}
 which holds for all $z\neq0$. Consequently, \eqref{eq:xi_eq_ratio_theta} is established for $q^{1/2}<|z|<q^{-1/2}$, however, by applying the Identity Principle for analytic functions,
 the validity of \eqref{eq:xi_eq_ratio_theta} extends to all $z\notin\{-q^{n+1/2}\mid n\in\Z\}\cup\{0\}$.
\end{proof}

\begin{prop}\label{prop:l2norm}
 For all $z\in\mathbb{C}\setminus\{0\}$, one has
 \begin{equation}
  \sum_{n=-\infty}^{\infty}\varphi_{n}^{2}(z)=4\frac{\left(q^{2};q^{2}\right)_{\infty}^{\!2}}{\left(q;q^{2}\right)_{\infty}^{2}}\theta_{q^{2}}\left(-z^{2}\right)\!.
 \label{eq:norm_varphi}
 \end{equation}
\end{prop}

\begin{proof}
 We start with the Green's formula:
 \begin{equation}
  (x-y)\sum_{k=m+1}^{n}\varphi_{k}(x)\varphi_{k}(y)=W_{n}\left(\varphi(x),\varphi(y)\right)-W_{m}\left(\varphi(x),\varphi(y)\right)
 \label{eq:Green_phi}
 \end{equation}
 holding for all $m,n\in\mathbb{Z}$, $m\geq n$, where
 \[
  W_{n}\left(\varphi(x),\varphi(y)\right)=q^{-n}\left(\varphi_{n+1}(x)\varphi_{n}(y)-\varphi_{n}(x)\varphi_{n+1}(y)\right)
 \]
 and $\varphi$ is given by \eqref{eq:phi_def}. Clearly, $W_{m}\left(\varphi(x),\varphi(y)\right)\to0$, for $m\to-\infty$,
 since $\varphi_{n}(x)\in\ell^{2}(\mathbb{Z})$, for all $x\neq0$, by Proposition \ref{prop:l2_sol}. Thus, by sending $m\to-\infty$ in \eqref{eq:Green_phi}, one finds
 \[
  (x-y)\sum_{k=-\infty}^{n}\varphi_{k}(x)\varphi_{k}(y)=W_{n}\left(\varphi(x),\varphi(y)\right)
 \]
 Further, sending $y\to x$ in the last identity, one arrives at the formula
 \begin{equation}
  \sum_{k=-\infty}^{\infty}\varphi_{k}^{2}(x)=\lim_{n\to\infty}W_{n}\left(\varphi'(x),\varphi(x)\right), \quad x\neq0.
 \label{eq:wronsk_to_norm}
 \end{equation}

 Thus, the evaluation of the sum of squares of $\varphi_{n}(x)$ is a matter of computation of the limit on the RHS of~\eqref{eq:wronsk_to_norm}. 
 By using formulas \eqref{eq:phi_def}, \eqref{eq:psipm_def} and \eqref{eq:varphi_2nd_expr}, it is not hard to see the convergence of the series on the LHS 
 of~\eqref{eq:wronsk_to_norm} is local uniform in $x\in\C\setminus\{0\}$. Consequently, the LHS of~\eqref{eq:wronsk_to_norm} is
 analytic on $\C\setminus\{0\}$. Since the function on the RHS of \eqref{eq:norm_varphi} is analytic in $x$ on $\C\setminus\{0\}$ as well, we may restrict 
 ourself with real $x$ and verify \eqref{eq:norm_varphi} for $x\in\R\setminus\{0\}$ only.	
 
 Let $x\in\R\setminus\{0\}$ be fixed. We temporarily denote $a:=\theta_{q}(-\ii q^{-1/2}x)$ and $\psi:=\psi^{(-)}$ Then, by \eqref{eq:phi_def}, 
 $\varphi=a\psi+\bar{a}\bar{\psi}$, where the bar denotes the complex conjugation.
 One has
 \[
  W_{n}\left(\varphi',\varphi\right)=a'aW_{n}\left(\psi,\psi\right)+a'\bar{a}W_{n}\left(\psi,\bar{\psi}\right)
  +a^{2}W_{n}\left(\psi',\psi\right)+a\bar{a}W_{n}\left(\psi',\bar{\psi}\right)+ \mbox{c.c.}
 \]
 where the abbreviation c.c. stands for the term which equals the complex conjugation to the previous four terms.
 Clearly, $W_{n}\left(\psi,\psi\right)=0$. Next, a straightforward computation shows
 \[
  \lim_{n\to\infty}W_{n}\left(\psi',\psi\right)=\lim_{n\to\infty}W_{n}\left(\psi',\bar{\psi}\right)=0.
 \]
 Recall at last that $W_{n}\left(\psi,\bar{\psi}\right)=-2\ii q^{1/2}$ as it is shown in the proof of Proposition \ref{prop:sol_qExp}. Hence, according to \eqref{eq:wronsk_to_norm}, one gets
 \[
  \sum_{k=-\infty}^{\infty}\varphi_{k}^{2}(x)=2\ii q^{1/2}\left(a(x)\bar{a}'(x)-a'(x)\bar{a}(x)\right)=
  -2\left(\theta_{q}\left(-\ii q^{-1/2}x\right)\theta_{q}'\left(\ii q^{-1/2}x\right)+\mbox{c.c.}\right)\!.
 \]
 Finally, it suffices to use the identity
 \[
  \theta_{q}'(z)\theta_{q}(-z)+\theta_{q}'(-z)\theta_{q}(z)=-2\theta_{q}(z)\theta_{q}(-z)\sum_{k\in\mathbb{Z}}\frac{q^{k}}{1-z^{2}q^{2k}},
 \]
 which holds true for all $z\neq0$, to arrive at the identity 
 \[
  \sum_{k=-\infty}^{\infty}\varphi_{k}^{2}(x)=4\theta_{q}\left(-\ii q^{-1/2}x\right)\theta_{q}\left(\ii q^{-1/2}x\right)\sum_{k\in\mathbb{Z}}\frac{q^{k}}{1+x^{2}q^{2k-1}}
  =4\theta_{q^{2}}\left(-q^{-1}x^{2}\right)\xi_{q^{2}}\left(x^{2}\right)
 \]
 where $\xi_{q}$ is defined in \eqref{eq:def_xi}. Now, it suffices to apply \eqref{eq:xi_eq_ratio_theta} to conclude the proof.
\end{proof}

\subsection{Self-adjoint extensions}

Recall that the symmetric operator $A$ decomposes as in \eqref{eq:A_decomp} and deficiency indices of $A_{+}$ are $(1,1)$, while deficiency indices of $A_{-}$ are $(0,0)$.
In another terminology, the difference expression associated with the operator $A$ is
in the limit circle case at $+\infty$ and in the limit point case at $-\infty$. Consequently, whether a sequence $\psi\in\ell^{2}(\mathbb{Z})$ belongs to the domain
of a self-adjoint extension of $A$, depends on the asymptotic behavior of $\psi$ at $+\infty$. 
Below, we provide the description of all self-adjoint extensions of $A$ in terms of boundary condition imposed on the respective sequences at $+\infty$. We use the advantage that,
in the case of $A$, we can restrict ourself to $A_{+}$ since, clearly, the set of all self-adjoint extensions of $A$ is in one-to-one correspondence with the set of all 
self-adjoint extensions of~$A_{+}$. 

Denote by $(A_{+})_{\max}$ the maximal domain Jacobi operator associated with $A_{+}$, hence
\[
 D_{+}:=\Dom(A_{+})_{\max}=\{\psi\in\ell^{2}(\Z_{+})\mid A_{+}\psi\in\ell^{2}(\Z_{+})\},
\]
and similarly $A_{\max}$ and $D$ in the case of $A$. It is known that all the self-adjoint extensions of $A_{+}$ can be parametrized by $t\in\mathbb{R}\cup\{\infty\}$
and their domains can be expressed as
\[
 \Dom(A_{+})_{t}=\{\psi\in\Dom(A_{+})_{\max}\mid \lim_{n\to\infty}(W_{n}(p,\psi)-tW_{n}(q,\psi))=0\},
\]
for $t\in\R$, and
\[
 \Dom(A_{+})_{\infty}=\{\psi\in\Dom(A_{+})_{\max}\mid \lim_{n\to\infty} W_{n}(q,\psi)=0\},
\]
where $p,q\in\ell^{2}(\Z_{+})$ can be chosen as the solutions of the equation
\[
 q^{-n+1}\psi_{n-1}+q^{-n}\psi_{n+1}=0, \quad  n\in\N,
\]
determined by the initial conditions $p_{0}=1$, $p_{1}=0$ and $q_{0}=0$, $q_{1}=1$, see, for example, \cite[Sec.~2.6]{Teschl}.
In our case, the entries of $p$ and $q$ are as in \eqref{eq:p_q_def}.

\begin{prop}\label{prop:dom_A_t}
 For $t\in\R\cup\{\infty\}$, operators $A_{t}\subset A_{\max}$ with domains
 \[
  \Dom A_{t}=\left\{\psi\in D \mid \lim_{n\to\infty}q^{-n}\left(\psi_{2n+1}+t\psi_{2n}\right)=0 \ \wedge \
  \lim_{n\to\infty}q^{-n}\left(q\psi_{2n-1}-t\psi_{2n}\right)=0\right\}\!,
 \]
 if $t\in\R$, or
 \[
  \Dom A_{\infty}=\left\{\psi\in D \mid \lim_{n\to\infty}q^{-n}\psi_{2n}=0\right\}\!,
 \]
 are all self-adjoint extensions of $A$.
\end{prop}

\begin{proof}
 According to the above introduction, it suffices to evaluate the Wronskians $W_{n}(p,\psi)$ and $W_{n}(q,\psi)$. 
 By using \eqref{eq:p_q_def}, one readily verifies that
 \[
  W_{2n}(p,\psi)-tW_{2n}(q,\psi)=(-1)^{n+1}q^{-n}\left(\psi_{2n+1}+t\psi_{2n}\right)
 \]
 and
 \[
  W_{2n-1}(p,\psi)-tW_{2n-1}(q,\psi)=(-1)^{n}q^{-n}\left(q\psi_{2n-1}-t\psi_{2n}\right)\!.
 \]
 Clearly, concerning the limit of the above expressions for $n\to\infty$, the alternating factors
 are superfluous and the statement follows.
\end{proof}

\subsection{The point spectrum of $A_{t}$}

We derive a secular equation whose zeros determine the point spectrum of a particular self-adjoint extension of $A$.

\begin{thm}\label{thm:secular_eq}
 For $t\in\R\cup\{\infty\}$, $\spec_{p}(A_t)$ coincides with the set of roots of the secular equation:
 \[
  x\theta_{q^{4}}\left(q^{2}x^{2}\right)+t\theta_{q^{4}}\left(x^{2}\right)=0, \mbox{ for } t\in\R, \; \mbox{ and } \; \theta_{q^{4}}\left(x^{2}\right)=0,  \mbox{ for } t=\infty.
 \]
 In addition, all eigenvalues of $A_t$ are simple.
\end{thm}

\begin{proof}
 First, note that $\spec_{p}(A_t)\subset\R\setminus\{0\}$ as it follows from the self-adjointness of $A_{t}$ and Proposition \ref{prop:spec_A_c_eq_0}.
 Next, according to Proposition \ref{prop:l2_sol}, $0\neq x\in\spec_{p}(A_t)$ if and only if the vector $\varphi(x)$, defined by \eqref{eq:phi_def}, 
 belongs to the domain of $A_t$. Taking into account Proposition \ref{prop:dom_A_t}, we find that
 $x\in\spec_{p}(A_{t})$ iff
 \[
  \lim_{n\to\infty}q^{-n}\left(\varphi_{2n+1}(x)+t\varphi_{2n}(x)\right)=0 \quad \mbox{ and }\quad
  \lim_{n\to\infty}q^{-n}\left(q\varphi_{2n-1}(x)-t\varphi_{2n}(x)\right)=0,
 \]
 for $t\in\R$ and similarly for $t=\infty$. By substituting for $\varphi(x)$ by \eqref{eq:phi_def} and \eqref{eq:psipm_def}, 
 the first boundary condition is equivalent to
 \[
 \lim_{n\to\infty}q^{1/2}\Im\eta_{2n+1}(x)+t\Re\eta_{2n}(x)=0
 \]
 where
 \[
  \eta_{n}(x):=\theta_{q}\left(-\ii q^{-1/2}x\right){}_{1}\phi_{1}\!\left(0;-q;q,\ii x q^{n+1/2}\right)\!.
 \]
 Since ${}_{1}\phi_{1}\!\left(0;-q;q,\ii x q^{n+1/2}\right)=1+o(1)$, as $n\rightarrow\infty$, we end up with the secular equation
 of the form
 \begin{equation}
  q^{1/2}\Im\theta_{q}\left(-\ii q^{-1/2}x\right)+t\Re\theta_{q}\left(-\ii q^{-1/2}x\right)=0.
 \label{eq:secular_eq_theta}
 \end{equation}
 The second boundary condition turns out to be equivalent to the same equation.
 
 By using the Jacobi triple product identity \eqref{eqdeftheta}, one deduces
 \[
  \Re\theta_{q}\left(-\ii q^{-1/2}x\right)=\frac{1}{(q;q)_{\infty}}\sum_{n\in\Z}q^{2n(n-1)}(-1)^{n}x^{2n}=\frac{(q^{4};q^{4})_{\infty}}{(q;q)_{\infty}}\theta_{q^{4}}\left(x^{2}\right)
 \]
 and
 \[
  \Im\theta_{q}\left(-\ii q^{-1/2}x\right)=\frac{xq^{-1/2}}{(q;q)_{\infty}}\sum_{n\in\Z}q^{2n^{2}}(-1)^{n}x^{2n}=xq^{-1/2}\frac{(q^{4};q^{4})_{\infty}}{(q;q)_{\infty}}\theta_{q^{4}}\left(q^{2}x^{2}\right)\!.
 \]
 By plugging the last formulas into \eqref{eq:secular_eq_theta}, one arrives at the secular equation from the statement for $t\in\R$.
 The secular equation for $t=\infty$ is to be derived in a completely analogous way.
 
 Finally, the simplicity of eigenvalues is a consequence of the part of Proposition~\ref{prop:l2_sol} 
 concerning uniqueness.
\end{proof}

Note that in two particular cases, namely when $t=0$ and $t=\infty$, roots of the secular equation can be found
fully explicitly. Thus, we have the following corollary.

\begin{cor}\label{cor:spec_point_0_inf}
 It holds
 \[
  \spec_{p}(A_0)=\left\{\pm q^{2n+1} \mid n\in\Z \right\}
 \quad
 \mbox{ and }
 \quad
  \spec_{p}(A_\infty)=\left\{\pm q^{2n} \mid n\in\Z \right\}.
 \]
\end{cor}

\subsection{Reparametrization and the spectrum fully explicitly}

It turns out that if a suitable reparametrization $t=\Phi(s)$ is applied, where $t$ comes from the parametrization used for the self-adjoint 
extensions $A_t$, the spectrum of $A_{\Phi(s)}$ can be expressed in a full explicit form in terms of the new parameter $s$. The crucial role here 
is played by the four theta functions:
\begin{equation}
\begin{aligned}
 \vartheta_{1}(z\mid q)&=\ii q^{1/4}e^{-\ii z}\left(q^{2};q^{2}\right)_{\!\infty}\theta_{q^{2}}\left(e^{2\ii z}\right)\!,\\
 \vartheta_{2}(z\mid q)&= q^{1/4}e^{-\ii z}\left(q^{2};q^{2}\right)_{\!\infty}\theta_{q^{2}}\left(-e^{2\ii z}\right)\!,\\
 \vartheta_{3}(z\mid q)&=\left(q^{2};q^{2}\right)_{\!\infty}\theta_{q^{2}}\left(-qe^{2\ii z}\right)\!,\\
 \vartheta_{4}(z\mid q)&=\left(q^{2};q^{2}\right)_{\!\infty}\theta_{q^{2}}\left(qe^{2\ii z}\right)\!;\\
\end{aligned}
\label{eq:def_theta1234}
\end{equation}
the closely related Jacobian elliptic functions and elliptic integrals. As a reference concerning these special functions, we primarily use the monograph \cite{Lawden}.

\begin{lem}
 The function
 \begin{equation}
  \Phi(s):=\ii q^{1/2}\frac{\vartheta_{4}\left(\ii s \mid q^{2}\right)}{\vartheta_{1}\left(\ii s \mid q^{2}\right)}
 \label{eq:def_Phi}
 \end{equation}
 is real-valued, strictly decreasing on $(0,-2\ln q)$, and maps $[0,-2\ln q)$ onto $\mathbb{R}\cup\{\infty\}$, with convention
 $\Phi(0)=\infty$. For the inverse function, it holds
 \begin{equation}
  \Phi^{-1}(t)=\frac{q^{1/2}}{\vartheta_{2}\left(0 \mid q^{2}\right)\vartheta_{3}\left(0 \mid q^{2}\right)}
  \int_{t}^{\infty}\left[\left(\frac{q\vartheta_{3}^{2}\left(0 \mid q^{2}\right)}{\vartheta_{2}^{2}\left(0 \mid q^{2}\right)}+x^{2}\right)
  \left(q+x^{2}\right)\right]^{-1/2}\!\!\!\dd x
 \label{eq:Phi_inverse}
 \end{equation}
 where, for $t=\infty$, the integral has to be understood as $0$.
\end{lem}

\begin{rem}
 The formula \eqref{eq:Phi_inverse} can be also written in the form without theta functions:
 \begin{equation*}
  \Phi^{-1}(t)=\frac{\left(q^{2};q^{2}\right)_{\infty}^{2}}{2\left(q^{4};q^{4}\right)_{\infty}^{4}}
  \int_{t}^{\infty}\left[\left(\frac{\left(-q^{2};q^{4}\right)_{\infty}^{4}}{4\left(-q^{4};q^{4}\right)_{\infty}^{4}}+x^{2}\right)
  \left(q+x^{2}\right)\right]^{-1/2}\!\!\!\dd x.
 \end{equation*}
\end{rem}

\begin{proof}
 It turns out that the function $\Phi$ can be expressed as a Jacobian elliptic function. Indeed, using definitions
 \cite[Eqs. 2.1.1, 2.1.21]{Lawden} and the Jacobi's imaginary transformation \cite[Eq. 2.6.12]{Lawden}, we arrive at the formula
 \[
  \Phi(s)=q^{1/2}\frac{\vartheta_{3}\left(0 \mid q^{2}\right)}{\vartheta_{2}\left(0 \mid q^{2}\right)}
  \cs\left(s\vartheta_{3}^{2}\left(0 \mid q^{2}\right), k'(q^{2})\right)
 \]
 where
 \[
  k'(q^{2})=\frac{\vartheta_{4}^{2}\left(0 \mid q^{2}\right)}{\vartheta_{3}^{2}\left(0 \mid q^{2}\right)},
 \]
 see \cite[Eq. 2.1.13]{Lawden}. Recall that $\cs(u \mid k')=\cn(u \mid k')/\sn(u \mid k')$ and $\cs(\cdot\mid k')$ is decreasing function
 on $\left(0,2K(k')\right)$ which maps the interval $\left(0,2K(k')\right)$ onto $\mathbb{R}$, where \cite[Eq. 2.2.3]{Lawden}
 \[
  K\left(k'(q^{2})\right)=-\ln(q)\, \vartheta_{3}^{2}\left(0 \mid q^{2}\right)\!,
 \]
 see \cite[Chp.~2]{Lawden} or \cite[Chp.~16]{AbramowitzStegun}. If, in addition, we set $\cs(0\mid k'):=\infty$,
 we verify the first part of the statement.
 
 The inverse to $\cs(\cdot\mid k')$ has the form of an elliptic integral, namely
 \[
  \cs^{-1}(v\mid k')=\int_{v}^{\infty}\frac{\dd t}{\sqrt{(1+t^{2})(k^{2}+t^{2})}}, \quad v\in\R,
 \]
 where
 \[
  k(q)=\sqrt{1-(k'(q))^{2}}=\frac{\vartheta_{2}^{2}\left(0 \mid q\right)}{\vartheta_{3}^{2}\left(0 \mid q\right)},
 \]
 as one obtains by putting $a=1$ and $b=k$ in \cite[Eq. 3.2.15]{Lawden}. With the aid of this integral formula, one arrives at \eqref{eq:Phi_inverse} in a routine way.
\end{proof}

\begin{thm}
 Let $t\in\mathbb{R}\cup\{\infty\}$, then
 \[
  \spec_{p}(A_t)=\left\{-e^{-s}q^{2m} \mid m\in\Z\right\}\cup\left\{e^{s}q^{2n} \mid n\in\Z\right\}
 \]
 where $s=\Phi^{-1}(t)$ is determined by \eqref{eq:Phi_inverse}.
\end{thm}

\begin{proof}
 By Theorem \ref{thm:secular_eq}, the set $\spec_{p}(A_t)$, for $t\in\R$, coincides with nonzero solutions of the equation
 \begin{equation}
  x\theta_{q^{4}}\left(q^{2}x^{2}\right)+t\theta_{q^{4}}\left(x^{2}\right)=0.
 \label{eq:secular_t_real}
 \end{equation}
 First, note that $x\neq0$ is a solution of \eqref{eq:secular_t_real} if and only if $-1/x$ is a solution of \eqref{eq:secular_t_real},
 as one easily verifies with the aid of \eqref{eq:theta_power_q_id} and \eqref{eq:theta_recip_arg}.
 
 Taking into account the symmetry $x\leftrightarrow-1/x$ of $\spec_{p}(A_t)$, we may restrict ourself to $x>0$ only. 
 Then, by putting $x=e^{-y}$, with $y\in\R$, and applying \eqref{eq:def_theta1234} we may rewrite \eqref{eq:secular_t_real} in the form
 \[
  \vartheta_{4}(\ii y\mid q^{2})-\ii q^{-1/2}t\vartheta_{1}(\ii y\mid q^{2})=0.
 \]
 Further, substituting for $t=\Phi(s)$ by \eqref{eq:def_Phi} and making use of the identity \cite[Eq.~1.4.8]{Lawden}
 \[
  \vartheta_{1}(u\mid q^{2})\vartheta_{4}(v\mid q^{2})+\vartheta_{1}(v\mid q^{2})\vartheta_{4}(u\mid q^{2})=
  \vartheta_{1}\left(\frac{u+v}{2}\,\Big|\, q\right)\vartheta_{2}\left(\frac{u-v}{2}\,\Big|\, q\right)\!,
 \]
 we obtain the secular equation in the form
 \[
  \vartheta_{1}\left(\ii\,\frac{s+y}{2}\,\Big|\, q\right)\vartheta_{2}\left(\ii\,\frac{s-y}{2}\,\Big|\, q\right)=0.
 \]
 Finally, with aid of \eqref{eq:def_theta1234}, one concludes that all the positive
 solutions of the secular equation have to satisfy the equation
 \[
  \theta_{q^{2}}\left(e^{-y-s}\right)=0,
 \]
 which are the numbers $x=e^{s}q^{2n}$, $n\in\Z$.
 
 The case $t=\infty$ has been treated in Corollary \ref{cor:spec_point_0_inf}.
\end{proof}

\begin{rem}\label{rem:eigenval_fill}
 Note that if $t$ ranges $(-\infty,\infty]$ then $e^{s}$ ranges the interval $[1,q^{-2})$.
 Thus, to every $\omega\neq0$, there exists a unique self-adjoint extension $A_{t}$ of $A$ with 
 \[
  \spec_{p}(A_t)=-\omega^{-1}q^{2\mathbb{Z}}\cup\omega q^{2\mathbb{Z}}.
 \]
 Indeed, the parameter $t$ is related to $\omega$ by the relation
 \[
  \exp\Phi^{-1}(t)=|\omega|q^{-2\lfloor\log_{q^{2}}|\omega|\rfloor}
 \]
 where $\lfloor x \rfloor$ denotes the largest integer less than or equal to $x\in\R$.
\end{rem}

\subsection{Consequences for the Ramanujan entire function}

Recall the Ramanujan entire function is defined as 
\[
 A_{q}(z):={}_{0}\phi_{1}\left(-;0;q,-qz\right)\!, \quad z\in\C.
\]
This function appeared repeatedly in the Ramanujan's ``Lost notebook'' \cite{Ramanujan} and represents one of possible 
$q$-analogues to the Airy function, see \cite{Ismail05}. By using the expression \eqref{eq:varphi_2nd_expr}, the entries 
of $\varphi(x)$ can be expressed in terms of the Ramanujan entire function as follows:
\begin{equation}
\varphi_{n}(x)=(-1;q)_{\infty} x^{n}q^{n(n-1)/2}A_{q^{2}}\left(q^{-2n+2}x^{-2}\right)
\label{eq:phi_Ramanujan}
\end{equation}
where $n\in\Z$ and $x\neq0$. 

Let us also remark that
\begin{equation}
 \|\varphi\left(x^{-1}\right)\|^{2}=x^{-2}\|\varphi(x)\|^{2},
\label{eq:norm_varphi_recip_arg}
\end{equation}
for $x\in\R\setminus\{0\}$, as one readily verifies by using \eqref{eq:norm_varphi} together with \eqref{eq:theta_recip_arg}.

\begin{prop}
 For any $\omega\in\R\setminus\{0\}$ and all $m,n\in\Z$ , one has
 \begin{equation}
  \sum_{k=-\infty}^{\infty}\varphi_{k}\left(\omega q^{2m}\right)\varphi_{k}\left(\omega q^{2n}\right)=\omega^{-4n}q^{-2n(2n-1)}\|\varphi\left(\omega\right)\|^{2}\,\delta_{m,n}
 \label{eq:OG_rel_varphi_first}  
 \end{equation}
 and
 \begin{equation}
  \sum_{k=-\infty}^{\infty}\varphi_{k}\left(\omega q^{2m}\right)\varphi_{k}\left(-\omega^{-1}q^{2n}\right)=0
 \label{eq:OG_rel_varphi_second}
 \end{equation}
 where the $\ell^{2}$-norm on the RHS of \eqref{eq:OG_rel_varphi_first} is given in Proposition \ref{prop:l2norm}.
\end{prop}

\begin{proof}
 For any $\omega\in\R\setminus\{0\}$, $\{\varphi(\omega q^{2m}) \mid m\in\Z\}\cup\{\varphi(-\omega^{-1}q^{2n}) \mid n\in\Z\}$ is the set of eigenvectors 
 of a self-adjoint operator $A_{t}$ for some $t$, see Remark \ref{rem:eigenval_fill}.
 The statement is nothing but the orthogonality of eigenvectors of a self-adjoint operator. Indeed, the second relation \eqref{eq:OG_rel_varphi_second} is immediate.
 The first relation \eqref{eq:OG_rel_varphi_first} is a consequence of the orthogonality relation
 \[
  \left\langle\varphi(\omega q^{2m}),\varphi(\omega q^{2n})\right\rangle=\|\varphi\left(\omega q^{2n}\right)\|^{2}\,\delta_{m,n}
 \]
 and the identity
 \begin{equation}
  \|\varphi\left(\omega q^{2n}\right)\|^{2}=\omega^{-4n}q^{-2n(2n-1)}\|\varphi\left(\omega\right)\|^{2}
 \label{eq:norm_varphi_power_q}
 \end{equation}
 which one deduces straightforwardly by using \eqref{eq:norm_varphi} and \eqref{eq:theta_power_q_id}.
 
\end{proof}

\begin{cor}
 For $z\in\C\setminus\{0\}$ and $\ell\in\Z$, the Ramanujan entire function satisfies the following orthogonality relations:
 \begin{equation}
  \sum_{k=-\infty}^{\infty}z^{k}q^{k(k-1)/2}A_{q}\left(zq^{k+\ell}\right)A_{q}\left(zq^{k-\ell}\right)=(q;q)_{\infty}^{2}\,\theta_{q}\left(-z\right)\delta_{0,\ell}
 \label{eq:OG_rel_Ramnujan_first}
 \end{equation}
 and
 \begin{equation}
  \sum_{k=-\infty}^{\infty}(-1)^{k}q^{k(k-1)/2}A_{q}\left(zq^{k+\ell}\right)A_{q}\left(z^{-1}q^{k-\ell}\right)=0.
 \label{eq:OG_rel_Ramnujan_second}
 \end{equation}
\end{cor}

\begin{proof}
 The first orthogonality relation \eqref{eq:OG_rel_Ramnujan_first} follows from \eqref{eq:OG_rel_varphi_first}. Indeed, by substituting in \eqref{eq:OG_rel_varphi_first} by~\eqref{eq:phi_Ramanujan},
 shifting the summation index by introducing $j=-k+1+m+n$, writing $-m$ instead of $m$, $-n$ instead of $n$, $q^{1/2}$ instead of $q$, and finally putting $z=\omega^{-2}$, one arrives at the formula
 \[
  \sum_{j=-\infty}^{\infty}z^{j}q^{j(j-1)/2}A_{q}\left(zq^{j+(m-n)}\right)A_{q}\left(zq^{j-(m-n)}\right)=
  \frac{4\left(q;q\right)_{\infty}^{2}}{\left(-1;q^{1/2}\right)_{\!\infty}^{\!2}\left(q^{1/2};q\right)_{\!\infty}^{\!2}}\theta_{q}\left(-z\right)\delta_{m,n}
 \]
 where we used \eqref{eq:norm_varphi} and \eqref{eq:norm_varphi_recip_arg}. This yields \eqref{eq:OG_rel_Ramnujan_first} for $z>0$, since $(-q;q)_{\infty}(q;q^{2})_{\infty}=1$.
 Similarly, one obtains the second relation \eqref{eq:OG_rel_Ramnujan_second} starting from \eqref{eq:OG_rel_varphi_second}, for $z>0$.
 
 It is very easy to verify the series on the LHSs of \eqref{eq:OG_rel_Ramnujan_first} and \eqref{eq:OG_rel_Ramnujan_second} both locally converge in $z$ on $\C\setminus\{0\}$ and 
 hence they are analytic functions in $z$ on $\C\setminus\{0\}$. In addition, the RHSs of \eqref{eq:OG_rel_Ramnujan_first} and \eqref{eq:OG_rel_Ramnujan_second} are functions analytic in $z$ on $\C\setminus\{0\}$, too. 
 Since \eqref{eq:OG_rel_Ramnujan_first} and \eqref{eq:OG_rel_Ramnujan_second} holds for all $z>0$, as we have already proved, their validity extends to all $z\in\C\setminus\{0\}$ by the Identity Principle 
 for analytic functions.
\end{proof}

At this point, we also mention the dual orthogonality relation and its consequences for the Ramanujan entire function.

\begin{prop}
 For any $\omega\in\R\setminus\{0\}$ and all $k,\ell\in\Z$ , one has
 \begin{equation}
  \sum_{n=-\infty}^{\infty}\omega^{2n}q^{n(n-1)}\left[\varphi_{k}\left(\omega q^{n}\right)\varphi_{\ell}\left(\omega q^{n}\right)+\varphi_{k}\left(-\omega^{-1}q^{-n+1}\right)\varphi_{\ell}\left(-\omega^{-1}q^{-n+1}\right)\right]
  =2\|\varphi\left(\omega\right)\!\|^{2}\delta_{k,\ell}
  \label{eq:OG_rel_varphi_third}
 \end{equation}
  where the $\ell^{2}$-norm on the RHS is given in Proposition \ref{prop:l2norm}.
\end{prop}

\begin{proof}
 For any $t\in\R\cup\{\infty\}$, we know the spectrum of the self-adjoint operator $A_{t}$ consists of isolated simple eigenvalues and zero which, in turn, is never an eigenvalue.
 Hence, with the aid of the Spectral Theorem, one infers that $A_{t}$ is an operator with pure point spectrum, i.e., its normalized eigenvectors form an orthonormal basis of $\ell^{2}(\Z)$.
 Recall that the set of eigenvectors of $A_{t}$ is of the form $\{\varphi(\omega q^{2m}) \mid m\in\Z\}\cup\{\varphi(-\omega^{-1}q^{2n}) \mid n\in\Z\}$, see Remark~\ref{rem:eigenval_fill}. 
  
 Thus, Parseval's equality yields
 \[
  \sum_{\lambda\in\omega q^{2\Z}}\frac{\varphi_{k}\left(\lambda\right)\varphi_{\ell}\left(\lambda\right)}{\|\varphi\left(\lambda\right)\|^{2}}+
  \frac{\varphi_{k}\left(-\lambda^{-1}\right)\varphi_{\ell}\left(-\lambda^{-1}\right)}{\|\varphi\left(-\lambda^{-1}\right)\|^{2}}=\delta_{k,\ell},
 \]
 for all $k,\ell\in\Z$. Taking into account \eqref{eq:norm_varphi_recip_arg} and \eqref{eq:norm_varphi}, the last relation can be written as
 \[
  \sum_{\lambda\in\omega q^{2\Z}}\frac{1}{\|\varphi\left(\lambda\right)\|^{2}}\left(\varphi_{k}\left(\lambda\right)\varphi_{\ell}\left(\lambda\right)+
  \lambda^{2}\varphi_{k}\left(-\lambda^{-1}\right)\varphi_{\ell}\left(-\lambda^{-1}\right)\right)=\delta_{k,\ell}.
 \]
 Further, by using \eqref{eq:norm_varphi_power_q}, one gets
 \begin{align*}
  \sum_{n=-\infty}^{\infty}\omega^{4n}q^{2n(2n-1)}\left[\varphi_{k}\left(\omega q^{2n}\right)\varphi_{\ell}\left(\omega q^{2n}\right)+\omega^{2}q^{4n}\varphi_{k}\left(-\omega^{-1}q^{-2n}\right)\varphi_{\ell}\left(-\omega^{-1}q^{-2n}\right)\right]&\\
  &\hskip-24pt=\|\varphi\left(\omega\right)\!\|^{2}\delta_{k,\ell}
 \end{align*}
 or equivalently
 \begin{align}
  \sum_{N\in2\Z}\omega^{2N}q^{N(N-1)}\left[\varphi_{k}\left(\omega q^{N}\right)\varphi_{\ell}\left(\omega q^{N}\right)+\omega^{2}q^{2N}\varphi_{k}\left(-\omega^{-1}q^{-N}\right)\varphi_{\ell}\left(-\omega^{-1}q^{-N}\right)\right]&\nonumber\\
  &\hskip-64pt=\|\varphi\left(\omega\right)\!\|^{2}\delta_{k,\ell}.
  \label{eq:OG_rel_varphi_even_inproof}
 \end{align}
 By writing $q\omega$ instead of $\omega$ in \eqref{eq:OG_rel_varphi_even_inproof} one also obtains the same relation with odd summation indices, i.e.,
 \begin{align}
  \sum_{N\in2\Z+1}\omega^{2N}q^{N(N-1)}\left[\varphi_{k}\left(\omega q^{N}\right)\varphi_{\ell}\left(\omega q^{N}\right)+\omega^{2}q^{2N}\varphi_{k}\left(-\omega^{-1}q^{-N}\right)\varphi_{\ell}\left(-\omega^{-1}q^{-N}\right)\right]&\nonumber\\
  &\hskip-64pt=\|\varphi\left(\omega\right)\!\|^{2}\delta_{k,\ell}
  \label{eq:OG_rel_varphi_odd_inproof}
 \end{align}
 where we used that $\|\varphi\left(q\omega\right)\!\|^{2}=\omega^{-2}\|\varphi\left(\omega\right)\!\|^{2}$ which follows from \eqref{eq:norm_varphi} and \eqref{eq:theta_power_q_id}.
 
 By summing up \eqref{eq:OG_rel_varphi_even_inproof} and \eqref{eq:OG_rel_varphi_odd_inproof}, one finds
 \[
  \sum_{n=-\infty}^{\infty}\omega^{2n}q^{n(n-1)}\left[\varphi_{k}\left(\omega q^{n}\right)\varphi_{\ell}\left(\omega q^{n}\right)+\omega^{2}q^{2n}\varphi_{k}\left(-\omega^{-1}q^{-n}\right)\varphi_{\ell}\left(-\omega^{-1}q^{-n}\right)\right]
  =2\|\varphi\left(\omega\right)\!\|^{2}\delta_{k,\ell},
 \]
 for all $k,\ell\in\Z$. Finally, to end with the relation \eqref{eq:OG_rel_varphi_third}, it suffices to split the sum in the last expression into two sums and shift the index by 1 in the second sum.
\end{proof}

\begin{cor}
  For $z\in\C\setminus\{0\}$ and $k,\ell\in\Z$, the Ramanujan entire function satisfies the following orthogonality relation:
 \begin{align}
  \sum_{n=-\infty}^{\infty}q^{n(n+k+\ell)}&\left(B_{n+k}(z)B_{n+\ell}(z)+B_{n+k}\!\left(-z^{-1}\right)B_{n+\ell}\!\left(-z^{-1}\right)\right)\nonumber\\
  &\hskip136pt=2q^{-k^{2}}\left(q^{2};q^{2}\right)_{\!\infty}^{\!2}\theta_{q^{2}}\left(-qz^{2}\right)\delta_{k,\ell}
 \label{eq:OG_rel_Ramnujan_third}
 \end{align}
 where 
 \[
  B_{j}(z)=z^{j}A_{q^{2}}\left(z^{2}q^{2j+1}\right)\!, \quad j\in\Z.
 \]
\end{cor}

\begin{proof}
 Again, we verify \eqref{eq:OG_rel_Ramnujan_third} only for $z\in\R\setminus\{0\}$. Then the validity of \eqref{eq:OG_rel_Ramnujan_third} is to be extended to all $z\in\C\setminus\{0\}$ by the analyticity argument.
 
 Introducing $z:=q^{1/2}\omega^{-1}$, one can rewrite \eqref{eq:OG_rel_varphi_third} as
 \begin{align*}
  \sum_{n=-\infty}^{\infty}q^{n^{2}}&\left[z^{-2n}\varphi_{k}\left(z^{-1}q^{n+1/2}\right)\varphi_{\ell}\left(z^{-1}q^{n+1/2}\right)+z^{2n}\varphi_{k}\left(-zq^{n+1/2}\right)\varphi_{\ell}\left(-zq^{n+1/2}\right)\right]\\
  &\hskip254pt=2\|\varphi\left(q^{1/2}z^{-1}\right)\!\|^{2}\delta_{k,\ell}.
 \end{align*}
 Then, substituting by \eqref{eq:phi_Ramanujan} into the last equation, using \eqref{eq:norm_varphi} and making some simple manipulations, one readily arrives at \eqref{eq:OG_rel_Ramnujan_third}.
\end{proof}

\begin{rem}
 Assume $k+\ell\in2\Z$, then, by writing $n-(k+\ell)/2$ instead of $n$, $z$ instead of $z^{2}$ and $q$ instead of $q^{2}$ in \eqref{eq:OG_rel_Ramnujan_third}, one gets the orthogonality relation
 \begin{align*}
  \sum_{n=-\infty}^{\infty}\!q^{n^{2}/2}&\left[z^{n}A_{q}\!\left(\!zq^{n+m+1/2}\right)\!A_{q}\!\left(\!zq^{n-m+1/2}\right)+z^{-n}A_{q}\!\left(\!z^{-1}q^{n+m+1/2}\right)\!A_{q}\!\left(\!z^{-1}q^{n-m+1/2}\right)\right]\\
  &\hskip226pt=2\left(q;q\right)_{\!\infty}^{\!2}\theta_{q}\left(-q^{1/2}z\right)\delta_{0,m},
 \end{align*}
 for all $z\in\C\setminus\{0\}$ and $m\in\Z$. However, this formula readily follows already from \eqref{eq:OG_rel_Ramnujan_first}.
 
 If $k+\ell\in2\Z+1$, then, by writing $n-(k+\ell+1)/2$ instead of $n$, $z$ instead of $z^{2}$ and $q$ instead of $q^{2}$ in \eqref{eq:OG_rel_Ramnujan_third}, one arrives at the relation
 \[
  \sum_{n=-\infty}^{\infty}\!q^{n(n-1)/2}\left[z^{n}A_{q}\!\left(zq^{n+m}\right)\!A_{q}\!\left(zq^{n-m-1}\right)-z^{1-n}A_{q}\!\left(z^{-1}q^{n+m}\right)\!A_{q}\!\left(z^{-1}q^{n-m-1}\right)\right]=0,
 \]
 for all $z\in\C\setminus\{0\}$ and $m\in\Z$.
\end{rem}

\section{Spectral analysis of operator $B$}
\label{sec:opB}

The linear operator $B=B(\alpha,q)$ determined by \eqref{eq:Be_n} is assumed to coincide with the discrete Schr{\" o}dinger operator
of the form
\[
 B=U+U^{*}+\alpha V
\]
where $U$ is the forward shift operator acting on $\ell^{2}(\mathbb{Z})$, i.e., $Ue_{n}=e_{n+1}$, for all $n\in\mathbb{Z}$, $U^{*}$ its adjoint, and 
$V$ is the self-adjoint operator determined by equalities $Ve_{n}=q^{-n}e_{n}$, for all $n\in\mathbb{Z}$. Since $U+U^{*}$ is bounded, the domain of $B$ 
coincides with the natural domain of $V$, i.e.,
\[
 \Dom B=\Dom V=\{\psi\in\ell^{2}(\Z)\mid V\psi\in\ell^{2}(\Z)\}.
\]
Note that $B$ is a sum of hermitian and self-adjoint operator, hence $B$ is self-adjoint.

If $\alpha=0$, then $B=U+U^{*}$ is an operator sometimes called discrete Laplacian on~$\Z$ or free discrete Schr{\" o}dinger operator. Its spectral properties are well-known. Recall
at least that
\[
 \spec(U+U^{*})=\spec_{ac}(U+U^{*})=[-2,2].
\]
The coupling parameter $\alpha$ is assumed to be real, however, it can be restricted even more. First, note that we may consider $\alpha>0$ since
$B(\alpha,q)$ and $-B(-\alpha,q)$ are unitarily equivalent. In addition, from \eqref{eq:Be_n}, it is clear that we may even assume, for example, $\alpha\in(q,1]$, otherwise it suffices 
to shift the index.

\subsection{The essential spectrum of $B$}

First, let us take a look at the essential part of the spectrum of $B$. In the following proof, for $\mathcal{A}\subset\mathbb{Z}$, we denote by $P_{\mathcal{A}}$ 
the orthogonal projection on the space spanned by $\{e_{n} \mid n\in\mathcal{A}\}$. Recall also $\Z_{+}=\{0,1,2,\dots\}$, $\N=\Z_{+}\setminus\{0\}$, and $\Z_{-}=-\Z_{+}$.

\begin{prop}\label{prop:ess_spec_B}
 It holds 
 \[
 \spec_{ess}(B)=[-2,2].
 \]
\end{prop}

\begin{proof}
 First, decompose $B$ as $B=D+R+X$ where $D=P_{\Z_{-}}(U+U^{*})P_{\Z_{-}}+P_{\N}\,\alpha VP_{\N}$, $R=P_{\N}(U+U^{*})P_{\N}$, and $X=B-D-R$.
 Clearly, $X$ is a compact operator, thus
 \[
  \spec_{ess}(B)=\spec_{ess}(D+R),
 \]
 by the Weyl theorem. Further, since
 $$R(D+\ii)^{-1}=P_{\mathbb{N}}(U+U^{*})(\alpha V+\ii)^{-1}P_{\mathbb{N}}$$
 is a compact operator, $R$ is $D$-compact, and, by the Weyl theorem for relatively compact symmetric perturbations, see, for example \cite[Thm.~9.9]{Weidmann}, we have
 \[
  \spec_{ess}(D+R)=\spec_{ess}(D).
 \]
 Finally, it suffices to note that 
 \[
 \spec_{ess}(D)=\spec_{ess}(P_{\mathbb{Z}_{-}}(U+U^{*})P_{\mathbb{Z}_{-}})=[-2,2].
 \]
 The last equality is a well-known fact.
\end{proof}

\subsection{Solutions of the eigenvalue equation and their asymptotics}

The main aim of this subsection is to provide formulas for solutions of the eigenvalue equation for operator $B$ in terms of basic hypergeometric series
and analyze their asymptotic behavior as the index tends to $\pm\infty$. This will be essential for the later detailed spectral analysis of $B$.

We distinguish two regions for the spectral variable. First, the case when the spectral variable ranges $\R\setminus[-2,2]$, i.e., is
outside the essential spectrum of $B$, and when is within the essential spectrum of $B$.
For this purpose, it is convenient to consider the eigenvalue equation for operator $B$ in the following form:
\begin{equation}
  v_{n-1}+\left(\alpha q^{-n}-\mu(z)\right)v_{n}+v_{n+1}=0,\quad n\in\Z,
\label{eq:eigen_eq_B}
\end{equation}
where $\mu(z)=z+z^{-1}$ is the Joukowski conformal map. Note that $\mu$ maps $\{z\in\C \mid 0<|z|<1\}$ bijectively onto $\C\setminus[-2,2]$ and 
$\{e^{\ii\theta} \mid \theta\in[0,\pi]\}$ bijectively onto $[-2,2]$.

Let us denote the regularized confluent basic hypergeometric ${}_{1}\phi_{1}$ series by
\begin{equation}
 {}_{1}\tilde{\phi}_{1}(a;b;q,z):=(b;q)_{\infty}\,{}_{1}\phi_{1}(a;b;q,z).
\label{eq:def_1phi1_regul} 
\end{equation}
This function is well defined and analytic in all variables $a,b,z\in\C$.
Two solutions of the equation \eqref{eq:eigen_eq_B} are $\{f_{n}\}_{n=-\infty}^{\infty}$ and $\{g_{n}\}_{n=-\infty}^{\infty}$ where
\begin{equation}
  f_{n}=f_{n}(z):=(-1)^{n}\alpha^{-n}q^{\frac{1}{2}n(n+1)}
  \,_{1}\tilde{\phi}_{1}(0;z^{-1}\alpha^{-1}q^{n+1};q,z\alpha^{-1}q^{n+1})
\label{eq:sol_f}
\end{equation}
and
\begin{equation}
  g_{n}=g_{n}(z):=z^{-n}\,_{1}\tilde{\phi}_{1}(0;z\alpha q^{1-n};q,qz^{2}).
 \label{eq:sol_g}
\end{equation}
This statement can be verified by somewhat lengthy but straightforward computation.
Functions \eqref{eq:sol_f} and \eqref{eq:sol_g} are well defined for all $z,\alpha\in\C\setminus\{0\}$.

\begin{rem}\label{rem:symmet_sol}
Due to the symmetry relation
\[
 {}_{1}\tilde{\phi}_{1}(0;a;q,z)={}_{1}\tilde{\phi}_{1}(0;z;q,a),
\]
see \cite[Prop.~2.1]{KoornwinderSwarttouw}, the function $f_{n}=f_{n}(z)$ is invariant under the exchange $z\leftrightarrow z^{-1}$. On the other hand,
the second solution $g_{n}=g_{n}(z)$ does not possess such a symmetry. Since the equation \eqref{eq:eigen_eq_B}
is also invariant under $z\leftrightarrow z^{-1}$, the sequence $\{\tilde{g}_{n}\}_{n=-\infty}^{\infty}$, where $\tilde{g}_{n}=\tilde{g}_{n}(z):=g_{n}(z^{-1})$, is another solution
of the equation \eqref{eq:eigen_eq_B}.
\end{rem}

\begin{prop}\label{prop:two_sol_asympt_Wronks}
 For the Wronskian $W(f,g)=f_{n+1}g_{n}-f_{n}g_{n+1}$, where $f$ and $g$ are defined in \eqref{eq:sol_f} and \eqref{eq:sol_g},
 one has
 \begin{equation}
    W(f,g)=-z^{-1}\theta_{q}\left(\alpha z\right)\!.
 \label{eq:Wronsk_f_g}\end{equation}
 Consequently, functions $f$ and $g$ are two linearly independent solutions of \eqref{eq:eigen_eq_B} if and only if $\alpha\neq0$ and 
 $z\notin\alpha^{-1}q^{\mathbb{Z}}\cup\{0\}$.
 Moreover, asymptotic relations for $f_{n}$ and $g_{n}$ are as follows:
 \begin{align}
  f_{n}&=(-1)^{n}\alpha^{-n}q^{\frac{1}{2}n(n+1)}\left[1+o(1)\right]\!,\label{eq:f_n_asympt_+inf}\\
  g_{n}&=(-1)^{n}\alpha^{n}q^{-\frac{1}{2}n(n-1)}\left[\theta_{q}\!\left(z^{-1}\alpha^{-1}\right)+o(1)\right]\!,\label{eq:g_n_asympt_+inf}
 \end{align}
 as $n\rightarrow\infty$, and
 \begin{align}
  f_{n}&=z^{n}\left[\frac{\theta_{q}(\alpha z)}{(z^{2};q)_{\infty}}+o(1)\right]\!,\label{eq:f_n_asympt_-inf_|x|<1}\\
  g_{n}&=z^{-n}\left[(qz^{2};q)_{\infty}+o(1)\right]\!,\label{eq:g_n_asympt_-inf}
 \end{align}
 as $n\rightarrow-\infty$, however, the formula \eqref{eq:f_n_asympt_-inf_|x|<1} is true under the additional assumption $0<|z|<1$.
\end{prop}

\begin{proof}
 Derivation of the asymptotic formula \eqref{eq:f_n_asympt_+inf} is a matter of straightforward
 computation using the very definitions \eqref{eq:sol_f} and \eqref{eq:def_1phi1_regul}. To verify \eqref{eq:g_n_asympt_+inf},
 it suffices to write $g_{n}$ in the form
 \[
  g_{n}=(-1)^{n}\alpha^{n}q^{-\frac{1}{2}n(n-1)}\sum_{k=0}^{\infty}\left(q^{-k}z^{-1}\alpha^{-1};q\right)_{n}\left(q^{k+1}z\alpha;q\right)_{\!\infty}\frac{q^{\frac{1}{2}k(k+1)+kn}}{(q;q)_{k}}(-1)^{k}z^{2k}
 \]
 where we used that
 \[
  \left(z\alpha q^{1-n+k};q\right)_{\infty}=(-1)^{n}\alpha^{n}z^{n}q^{-\frac{1}{2}n(n-1)+kn}\left(q^{-k}z^{-1}\alpha^{-1};q\right)_{n}\left(q^{k+1}z\alpha;q\right)_{\!\infty}.
 \]
 Expansion \eqref{eq:g_n_asympt_+inf} now immediately follows.
 
 Having the expansions \eqref{eq:f_n_asympt_+inf} and \eqref{eq:g_n_asympt_+inf}, 
 the formula for the Wronskian \eqref{eq:Wronsk_f_g} follows from the independence of  $W(f,g)$ on $n$ using 
 \[
 W(f,g)=\lim_{n\rightarrow\infty}\left(f_{n+1}g_{n}-f_{n}g_{n+1}\right)
 \]
 together with the identity \eqref{eq:theta_recip_arg}.
 
 The asymptotic formula \eqref{eq:g_n_asympt_-inf} readily follows from definitions \eqref{eq:sol_g} and \eqref{eq:def_1phi1_regul} 
 with the aid of the identity \cite[Eq.~(1.3.16)]{GasperRahman}
 \[
 {}_{0}\phi_{0}(-;-;q,z)=(z;q)_{\infty}
 \]
 which holds for all $z\in\mathbb{C}$. 
 
 Finally, assuming $n<0$ and using that
 \[
  \left(z^{-1}\alpha^{-1}q^{1+n+k};q\right)_{\infty}=(-1)^{n}\alpha^{n}z^{n}q^{-\frac{1}{2}n(n+1)-kn}\left(q^{-k}z\alpha;q\right)_{-n}\left(q^{k+1}z^{-1}\alpha^{-1};q\right)_{\!\infty},
 \]
 we can express $f_{n}$ in the form
 \[
  f_{n}=z^{n}\sum_{k=0}^{\infty}\left(q^{-k}z\alpha;q\right)_{-n}\left(q^{k+1}z^{-1}\alpha^{-1};q\right)_{\!\infty}\frac{q^{\frac{1}{2}k(k+1)}}{(q;q)_{k}}(-1)^{k}\alpha^{-k}z^{k}.
 \]
 Hence
 \[
  \lim_{n\to\infty}z^{-n}f_{n}=\sum_{k=0}^{\infty}\theta_{q}\left(q^{-k}z\alpha\right)\frac{q^{\frac{1}{2}k(k+1)}}{(q;q)_{k}}(-1)^{k}\alpha^{-k}z^{k}.
 \]
 To evaluate the sum on the RHS in the last equality and arrive at \eqref{eq:f_n_asympt_-inf_|x|<1}, one has to use \eqref{eq:theta_power_q_id} first and then apply the identity \cite[Eq.~(1.3.15)]{GasperRahman}
 \begin{equation}
 {}_{1}\phi_{0}(0;-;q,z)=\frac{1}{(z;q)_{\infty}}
 \label{eq:1phi0_eq_qpoch}
 \end{equation}
 which holds true for $|z|<1$ only. This is the reason for additional assumption concerning the validity of \eqref{eq:f_n_asympt_-inf_|x|<1}.
\end{proof}

To complete the picture of the asymptotic behavior of the solutions \eqref{eq:sol_f} and \eqref{eq:sol_g}, we need to investigate the 
asymptotic behavior of $f_{n}(e^{\ii\phi})$, as $n\rightarrow-\infty$, for $\phi\in[0,\pi]$. This is, however, a bit more delicate task. 
Recall that, for the upcoming spectral analysis of $B$, it is sufficient to restrict the range of parameter $\alpha$ to $(q,1]$. This is already considered 
in the following lemma.

\begin{lem}\label{lem:asympt_aux}
The following asymptotic formulas hold true for $n\to\infty$.
\begin{enumerate}[{\upshape i)}]
  \item If $|z|=1$, $\Im z>0$, and $\xi\in[q,1)$, then
  \[
   {}_{1}\tilde{\phi}_{1}\left(0;z^{-1}\xi q^{-n};q,z\xi q^{-n}\right)=(-z\xi)^{n}q^{-\frac{1}{2}n(n+1)}
   \left[A+z^{-2n}B+o(1)\right]
  \]
  where
  \[
   A=\frac{\left(z^{-1}\xi;q\right)_{\infty}}{(q;q)_{\infty}}\left(\frac{z^{2}}{z^{2}-1}{}_{1}\phi_{1}\left(q;qz^{-2};q,qz^{-1}\xi^{-1}\right)
   +{}_{1}\phi_{1}\left(q;z^{-1}\xi;q,z\xi\right)-1\right)
  \]
  and
  \[
   B=\frac{\theta_{q}\!\left(z^{-1}\xi\right)}{\left(z^{2};q\right)_{\infty}}.
  \]
 \item
 \[
  {}_{1}\tilde{\phi}_{1}(0;\xi q^{-n};q,\xi q^{-n})=\begin{cases}
                                                     (-\xi)^{n}q^{-\frac{1}{2}n(n+1)}
                                                     \left[\frac{\theta_{q}\left(\xi\right)}{\left(q;q\right)_{\infty}}n+O(1)\right]\!, \; & \xi\in(q,1),\\
                                                     (-1)^{n}q^{-\frac{1}{2}n(n-1)}\left[\left(q;q\right)_{\infty}+o(1)\right]\!, \; & \xi=q.
                                                    \end{cases}
 \]
 \item If $\xi\in[q,1)$, then
 \[
  {}_{1}\tilde{\phi}_{1}(0;-\xi q^{-n};q,-\xi q^{-n})=\xi^{n}q^{-\frac{1}{2}n(n+1)}
                                                     \left[\frac{\theta_{q}\left(-\xi\right)}{\left(q;q\right)_{\infty}}n+O(1)\right]\!.
 \]
\end{enumerate}
\end{lem}

\begin{proof}
By the definition \eqref{eq:def_1phi1_regul}, one has
$${}_{1}\tilde{\phi}_{1}\left(0;z^{-1}\xi q^{-n};q,z\xi q^{-n}\right)=\sum_{k=0}^{\infty}\frac{\left(z^{-1}\xi q^{-n+k};q\right)_{\infty}}{(q;q)_{k}}q^{\binom{k}{2}-nk}(-z\xi)^{k}.$$
Next, we split the sum into two sums where the summation index of the first sum ranges from $0$ to $n-1$, while the index of the second one
goes from $n$ to $\infty$. Further, we make a substitution in the first sum by introducing new summation index $j:=n-k$ and
we shift the summation index of the second sum by $n$. Finally, by making some obvious algebraic manipulations, one arrives
at the expression
\begin{equation}
(-z\xi)^{n}q^{-\frac{1}{2}n(n+1)}\left(z^{-1}\xi;q\right)_{\infty}\!
\left[\sum_{j=1}^{n}\frac{(qz\xi^{-1};q)_{j}}{(q;q)_{n-j}}z^{-2j}
+\sum_{k=0}^{\infty}\frac{q^{k(k-1)/2}}{\left(z^{-1}\xi;q\right)_{k}(q;q)_{k+n}}(-z\xi)^{k}\right]\!\!.
\label{eq:tow_asym_aux}\end{equation}
Note the second sum in the square brackets in \eqref{eq:tow_asym_aux} has a finite limit when 
$n\rightarrow\infty$ equal to
\begin{equation}
 \frac{1}{(q;q)_{\infty}}\,_{1}\phi_{1}\left(q;z^{-1}\xi;q,z\xi\right)\!.
\label{eq:tow_asym_second_term_limit}
\end{equation}

Thus, in the following part, we investigate the asymptotic behavior of the first sum.
Let us denote
\begin{equation}
 P_{n}=P_{n}(z,\xi;q):=\sum_{j=0}^{n}\frac{(qz\xi^{-1};q)_{j}}{(q;q)_{n-j}}z^{-2j}.
 \label{eq:def_P_n}
\end{equation}
Note the summation index ranges $0\leq j \leq n$, while the index of the first sum in \eqref{eq:tow_asym_aux}
ranges $1\leq j \leq n$.

Next, we derive the generating function formula for $P_{n}$ given by \eqref{eq:def_P_n}. For $t\in\C$, $|t|<1$, one has
\[
 \sum_{n=0}^{\infty}P_{n}t^{n}=\sum_{l=0}^{\infty}\frac{t^{l}}{(q;q)_{l}}\sum_{k=0}^{\infty}\left(qz\xi^{-1};q\right)_{k}z^{-2k}t^{k}
 =\frac{1}{(t;q)_{\infty}}\,_{2}\phi_{1}\left(qz\xi^{-1},q;0;q,z^{-2}t\right)
\]
where the identity \eqref{eq:1phi0_eq_qpoch} was used. A limit case
of Heine's transformation formulas for $\,_{2}\phi_{1}$, see \cite[Sec.~1.4]{GasperRahman}, yields
\[
 \,_{2}\phi_{1}\left(a,b;0;q,z\right)=\frac{(bz;q)_{\infty}}{(z;q)_{\infty}}\,_{1}\phi_{1}\left(b;bz;q,az\right)\!.
\]
With the aid of this formula, we arrive at the generating function formula of the form
\begin{equation}
 \sum_{n=0}^{\infty}P_{n}t^{n}=\frac{1}{(1-z^{-2}t)(t;q)_{\infty}}\,_{1}\phi_{1}\left(q;qz^{-2}t;q,qz^{-1}\xi^{-1}t\right)=:\Psi(t).
\label{eq:gener_P_n}
\end{equation}
This formula is suitable for the application of Daurboux's method, see \cite[Section~8.9]{Olver}. Note that, if $|z|=1$, the singularities of $\Psi$
located most closely to the origin are $t=1$ and $t=z^{2}$.

Proof of assertion (i): If $|z|=1$, $\Im z>0$, the function $\Psi$ has two simple poles
located on the unit circle at $t=1$ and $t=z^{2}$. Thus, the application of Darboux's method with the comparison function
\[
 h(t):=\frac{a}{t-1}+\frac{b}{t-z^{2}}=\sum_{n=0}^{\infty}\left(-a-\frac{b}{z^{2n+2}}\right)t^{n},
\]
where
\[
 a=\Res(\Psi,1)=-\frac{1}{(1-z^{-2})(q;q)_{\infty}}\,_{1}\phi_{1}\left(q;qz^{-2};q,qz^{-1}\xi^{-1}\right)
\]
and
\[
 b=\Res\left(\Psi,z^{2}\right)=-\frac{z^{2}}{\left(z^{2};q\right)_{\infty}}\,_{0}\phi_{0}\left(-;-;q,qz\xi^{-1}\right)
  =-z^{2}\frac{\left(qz\xi^{-1};q\right)_{\infty}}{\left(z^{2};q\right)_{\infty}},
\]
results in the asymptotic formula
\[
 P_{n}=-a-\frac{b}{z^{2n+2}}+o(1), \quad \mbox{ as } n\to\infty.
\]
Last equation together with \eqref{eq:tow_asym_aux}, \eqref{eq:tow_asym_second_term_limit} and \eqref{eq:def_P_n} yields the claim (i).

Proof of assertion (ii): Let $\xi\in(q,1)$. For $z=1$, the function $\Psi$ has one pole at $t=1$ of order $2$.
The application of Darboux's method with the comparison function
\[
  h(t)=\frac{a}{(t-1)^{2}}+\frac{b}{t-1}=\sum_{n=0}^{\infty}\left(a(n+1)-b\right)t^{n}
\]
where
\[
 a=\lim_{t\to1}(t-1)^{2}\Psi(t)=\frac{1}{\left(q;q\right)_{\infty}}\,_{0}\phi_{0}\left(-;-;q,q\xi^{-1}\right)
  =\frac{\left(q\xi^{-1};q\right)_{\infty}}{\left(q;q\right)_{\infty}}
\]
and 
\[
 b=\Res\left(\Psi,1\right),
\]
yields
\[
 P_{n}=a(n+1)-b+o(1)=an+O(1), \quad \mbox{ as } n\to\infty.
\]
This formula together with \eqref{eq:tow_asym_aux}, \eqref{eq:tow_asym_second_term_limit} and \eqref{eq:def_P_n} implies
the expansion from the claim (ii) for $\xi\in(q,1)$.

If $\xi=q$ and $z=1$, then the first sum in \eqref{eq:tow_asym_aux} vanishes. Thus, it suffices to note that
\eqref{eq:tow_asym_second_term_limit} equals $1$, to obtain the second asymptotic expansion from the claim (ii).

Proof of assertion (iii): This derivation is again based on the method of Darboux applied to \eqref{eq:gener_P_n}
with $z=-1$ and $\xi\in[q,1)$. It follows similar steps as in the proof of the assertion (ii) with the only exception
that the first sum in \eqref{eq:tow_asym_aux} never vanishes for $z=-1$.
\end{proof}

\begin{prop}\label{prop:asympt_f_circle}
 The following asymptotic formulas hold true for $n\to-\infty$.
\begin{enumerate}[{\upshape i)}]
 \item If $\phi\in(0,\pi)$ and $\alpha\in(q,1]$, then
 \[
  f_{n}(e^{i\phi})=e^{-\ii n\phi}A+e^{\ii n\phi}B+o(1)
 \]
 where
   \begin{align*}
   A=&\frac{\left(q\alpha^{-1}e^{-\ii\phi};q\right)_{\infty}}{(q;q)_{\infty}}\\
   &\times\left(-\frac{\ii e^{\ii\phi}}{2\sin\phi}\,_{1}\phi_{1}\left(q;qe^{-2i\phi};q,\alpha e^{-\ii\phi}\right)
   +{}_{1}\phi_{1}\left(q;q\alpha^{-1}e^{-\ii\phi};q,q\alpha^{-1}e^{\ii\phi}\right)-1\right)
  \end{align*}
  and
  \[
   B=\frac{\theta_{q}\!\left(\alpha e^{\ii\phi}\right)}{\left(e^{2\ii\phi};q\right)_{\infty}}.
  \]
 \item If $\alpha\in(q,1)$, then
 \[
  f_{n}(1)=\frac{\theta_{q}(\alpha)}{\left(q;q\right)_{\infty}}n+O(1),
 \]		     
 while for $\alpha=1$, one has
 \[
  f_{n}(1)=(q;q)_{\infty}+o(1).
 \]

 \item If $\alpha\in(q,1]$, then
 \[
  f_{n}(-1)=\frac{\theta_{q}(-\alpha)}{\left(q;q\right)_{\infty}}n+O(1).
 \]
\end{enumerate}
\end{prop}

\begin{proof}
 The proof readily follows from the formulas given in Lemma \ref{lem:asympt_aux} by writing $-n$ instead of $n$, 
 $\xi=q\alpha^{-1}$ and $z=e^{\ii\phi}$, and the definition \eqref{eq:sol_f}.
\end{proof}

\begin{prop}\label{prop:no_sol_ell2_embedded}
 For all $\alpha\in\R$ and all $x\in[-2,2]$, there is no non-trivial solution of the second-order difference equation
 \[
  v_{n-1}+\left(\alpha q^{-n}-x\right)v_{n}+v_{n+1}=0, \quad n\in\Z,
 \]
 belonging to $\ell^{2}(\Z)$.
\end{prop}

\begin{proof}
 For $\alpha=0$, the statement is a well-known fact. If $\alpha\neq0$, it suffices to verify the statement for $\alpha\in(q,1]$.
 
 First, assume $(\alpha,x)\in(q,1]\times[-2,2]\setminus\{(1,2)\}$. If we put $x=\mu(z)$ with $z=e^{\ii\phi}$, $\phi\in[0,\pi]$,
 then $(\alpha,\phi)\in(q,1]\times[0,\pi]\setminus(1,0)$ and, according to Proposition \ref{prop:two_sol_asympt_Wronks}, the solutions 
 $f(e^{\ii\phi})$ and $g(e^{\ii\phi})$ are linearly independent. Consequently, any solution of \eqref{eq:eigen_eq_B} 
 is a linear combination of $f(e^{\ii\phi})$ and $g(e^{\ii\phi})$. However, Proposition \ref{prop:asympt_f_circle} 
 and relation \eqref{eq:g_n_asympt_-inf} yield nor $f(e^{\ii\phi})$, nor $g(e^{\ii\phi})$, are square summable at $-\infty$.
 
 In the remaining case, when $\alpha=1$ and $x=2$, the solutions $f(1)$ and $g(1)$ are linearly dependent.
 $f(1)$ is not square summable at $-\infty$ by claim (ii) of Proposition \ref{prop:asympt_f_circle}. 
 Let $y$ denotes the second linearly independent solution of (\ref{eq:eigen_eq_B}) with $z=\alpha=1$. Then 
 the Wronskian $f_{n+1}(1)y_{n}-f_{n}(1)y_{n+1}$ is a nonzero constant. However, $f_{n}(1)$ tends to zero 
 if $n\rightarrow\infty$,  by (\ref{eq:f_n_asympt_+inf}). It follows, the solution $y$ can not be bounded at $+\infty$. 
 Consequently, no nontrivial linear combination of $f(1)$ and $y$ does belong to $\ell^{2}(\Z)$. 
\end{proof}

\subsection{Spectrum of $B$}

Recall that we already know, by Proposition \ref{prop:ess_spec_B}, that $\spec_{ess}(B)=[-2,2]$.
In addition, as it follows from Proposition \ref{prop:no_sol_ell2_embedded}, the operator $B$ 
has no eigenvalue embedded in the essential spectrum. In this subsection, we will show that there infinitely many simple 
eigenvalues located outside the essential spectrum and they will be determined fully explicitly. Moreover, we will find 
the corresponding eigenvectors and compute their norm.

\begin{prop}\label{prop:spec_B}
 If $\alpha\neq0$, then
 \begin{equation*}
 \spec(B)\setminus[-2,2]=\spec_{p}(B)=\left\{\alpha^{-1}q^{m}+\alpha q^{-m} \mid m>\lfloor\log_{q}|\alpha|\rfloor\right\}
 \end{equation*}
 and all points from this set are simple eigenvalues of $B$. In addition, eigenvectors $\textbf{v}_{m}$ 
 corresponding to eigenvalues $\alpha^{-1}q^{m}+\alpha q^{-m}$, $m>\lfloor\log_{q}|\alpha|\rfloor$, can be chosen as 
 $\textbf{v}_{m}=\{v_{m,j}\}_{j=-\infty}^{\infty}$, with
 \begin{equation}
 v_{m,j}=(-1)^{j}\alpha^{-j}q^{\frac{1}{2}j(j+1)}\,_{1}\tilde{\phi}_{1}\left(0;q^{-m+j+1};q,\alpha^{-2}q^{m+j+1}\right)\!.
 \label{eq:eigenvec_v_mj}
 \end{equation}
\end{prop}

\begin{proof}
 Proposition \ref{prop:no_sol_ell2_embedded} tells us that $\Ker(B-x)=\{0\}$ for all $x\in[-2,2]$. By taking into account also Proposition
 \ref{prop:ess_spec_B}, one arrives at the equality $\spec(B)\setminus[-2,2]=\spec_{p}(B)$. Next, we divide the proof into 3 parts.
 
 1) Let $\alpha\in(q,1]$. Take $0<|z|<1$ such that $z\notin\alpha^{-1}q^{\N}$. By Proposition \ref{prop:two_sol_asympt_Wronks}, 
 any solution $y$ of the equation \eqref{eq:eigen_eq_B} is a linear combination of the solutions $f$ and $g$. In addition, by inspection of 
 the asymptotic behavior of $f_{n}$ and $g_{n}$, for $n\to\pm\infty$, one finds $y$ can not be in $\ell^{2}(\mathbb{Z})$, unless $y\neq0$. 
 Thus, we have $\Ker(B-x)=\{0\}$ for all $x\notin\alpha^{-1}q^{\N}+\alpha q^{-\N}$.
 
 On the other hand, if $z\in\alpha^{-1}q^{\N}$, then the solutions $f$ and $g$ are linearly dependent since their Wronskian 
 vanishes. Moreover, $f_{n}$ is square summable at $+\infty$ and $g_{n}$ is square summable at $-\infty$. Hence, $f\in\Ker(B-x)$ for all 
 $x\in\alpha^{-1}q^{\N}+\alpha q^{-\N}$. Consequently, points from the set $\alpha^{-1}q^{\N}+\alpha q^{-\N}$ are all eigenvalues of $B$ 
 and there are no other eigenvalues of $B$.
 Moreover, the eigenvector $\textbf{\textit{v}}_{m}$ corresponding to $\alpha^{-1}q^{m}+\alpha q^{-m}$, $m\in\N$, can be chosen such that 
 \begin{equation}
  v_{m,j}=f_{j}\left(\alpha^{-1}q^{m}\right)\!, \quad \forall j\in\Z,\, m\in\N.
 \label{eq:v_mj_rel_f}
 \end{equation}
 
 2) Let $\alpha>0$. Then there exists unique $\Delta\in\Z$ and $\beta\in(q,1]$ such that $\alpha=q^{\Delta}\beta$. Namely,
 \[
  \Delta=\lfloor\log_{q}\alpha\rfloor \quad \mbox{ and } \quad \beta=\alpha q^{-\lfloor\log_{q}\alpha\rfloor}.
 \]
 Let $U_{\Delta}$ stands for the unitary operator determined by equalities $U_{\Delta}e_{n}=e_{n+\Delta}$, for all $n\in\Z$.
 Then one has
 \begin{equation}
  U_{\Delta}^{*}B(\alpha)U_{\Delta}=B(\beta).
 \label{eq:Unit_Delta_B}
 \end{equation}
 Consequently, taking into account the already proved part 1), we obtain
 \[
  \spec_{p}(B(\alpha))=\spec_{p}(B(\beta))=\left\{\beta^{-1}q^{n}+\beta q^{-n} \mid n\in\N\right\}=\left\{\alpha^{-1}q^{m}+\alpha q^{-m} \mid m>\Delta\right\}.
 \]
 Further, by using \eqref{eq:Unit_Delta_B} and \eqref{eq:v_mj_rel_f}, one finds the $j$-th element of the eigenvector of $B(\alpha)$ corresponding to the eigenvalue
 $\alpha^{-1}q^{m}+\alpha q^{-m}$, $m>\Delta$, equals
 \begin{align}
  &\left(U_{\Delta}f\!\left(\alpha q^{-\Delta};\alpha^{-1}q^{m}\right)\right)_{j}=f_{j-\Delta}\!\left(\alpha q^{-\Delta};\alpha^{-1}q^{m}\right)\nonumber\\
  &\hskip20pt=(-1)^{j+\Delta}\alpha^{\Delta-j}q^{-\frac{1}{2}\Delta(\Delta+1)+\frac{1}{2}j(j+1)}\,_{1}\tilde{\phi}_{1}\left(0;q^{-m+j+1};q,\alpha^{-2}q^{m+j+1}\right) \label{eq:eigenvec_v_mj_in_proof_part_2}
 \end{align}
 where we temporarily designated the dependence of $f_{n}(z)$, given by \eqref{eq:sol_f}, on $\alpha$ by writing $f_{n}(\alpha;z)$. To arrive at \eqref{eq:eigenvec_v_mj}, it suffices 
 to omit the multiplicative constant not depending on~$j$ in \eqref{eq:eigenvec_v_mj_in_proof_part_2}.
 
 3) It remains to discuss the case $\alpha<0$. However, the verification of the statement straightforwardly follows from the already proven part 2) and the fact that
 \[
  V^{*}B(\alpha)V=-B(-\alpha)
 \]
 where $V$ is the unitary operator determined by equalities $Ve_{n}=(-1)^{n}e_{n}$, for all $n\in\Z$.

 Finally, the simplicity of eigenvalues is a general property which follows from the following. 
 If $u$ and $v$ are two linearly independent solutions of \eqref{eq:eigen_eq_B} both from $\ell^{2}(\mathbb{Z})$, their Wronskian 
 $W(u,v)=u_{n+1}v_{n}-u_{n}v_{n+1}$ is a nonzero constant. However, at the same time, this expression tends to $0$, as $n\rightarrow\infty$, which would lead to a contradiction.
\end{proof}

Next, we compute the norm of eigenvectors \eqref{eq:eigenvec_v_mj}.

\begin{prop}\label{prop:norm_eigenvec_B}
 Let $\alpha\neq0$. The eigenvectors $\textbf{\textit{v}}_{m}=\{v_{m,j}\}_{j=-\infty}^{\infty}\,$, whose components are given by \eqref{eq:eigenvec_v_mj},
 are normalized as follows:
 \begin{equation}
  \|\textbf{\textit{v}}_{m}\|_{\ell^{2}(\mathbb{Z})}=\frac{|\alpha|^{-m}q^{m(m+1)/2}}{\sqrt{1-\alpha^{-2}q^{2m}}}\,(q;q)_{\infty}, \quad m>\lfloor\log_{q}|\alpha|\rfloor.
 \label{eq:norm_eigenvec_v_n}
 \end{equation}
\end{prop}

\begin{proof}
 The proof is divided into two parts. First, we derive a general formula for the norm of an eigenvector which is subsequently applied to our specific case. 

 Recall $f(x)=\{f_{n}(x)\}_{n=-\infty}^{\infty}$ and $g(x)=\{g_{n}(x)\}_{n=-\infty}^{\infty}$ are solutions of \eqref{eq:eigen_eq_B} and, for $0<|x|<1$,
 $f(x)$ is square summable at $+\infty$ and $g(x)$ is square summable at $-\infty$, as it follows from \eqref{eq:f_n_asympt_+inf} and \eqref{eq:g_n_asympt_-inf}.
 By using the Green's formula, we obtain 
 \[
 \left(\mu(x)-\mu(y)\right)\sum_{j=m+1}^{n}g_{j}(x)g_{j}(y)=W_{n}(g(x),g(y))-W_{m}(g(x),g(y)).
 \]
 Since $g(x)$ is square summable at $-\infty$, for $0<|x|<1$, $W_{m}(g(x),g(y))\rightarrow0$, as $m\rightarrow-\infty$, if $|x|<1$ and $|y|<1$.
 Thus, by sending $m\to-\infty$ in the last equation, we arrive at the identity
 \[
 \left(\mu(x)-\mu(y)\right)\sum_{j=-\infty}^{n}g_{j}(x)g_{j}(y)=W_{n}(g(x),g(y)).
 \]
 Next, we divide both sides by $x-y$ and send $y\to x$ in the last equation. 
 Since the sum on the LHS converges locally uniformly with the arguments $x$ and $y$ in the unit disk, we can interchange the limit and the sum and
 we get
 \[
  \mu'(x)\sum_{j=-\infty}^{n}g_{j}^{2}(x)=W_{n}(g'(x),g(x)), \quad \mbox{ for } 0<|x|<1,
 \]
 which is equivalent to
 \begin{equation}
 \sum_{j=-\infty}^{n}g_{j}^{2}(x)=\frac{x^{2}}{x^{2}-1}W_{n}(g'(x),g(x)), \quad \mbox{ for } 0<|x|<1.
 \label{eq:to_ell_2_norm_g}
 \end{equation}
 Similarly, since $f(x)$ is square summable at $+\infty$, by \eqref{eq:f_n_asympt_+inf}, one deduces that
 \begin{equation}
 \sum_{j=m+1}^{\infty}f_{j}^{2}(x)=-\frac{x^{2}}{x^{2}-1}W_{m}(f'(x),f(x)), \quad \mbox{ for } 0<|x|<1.
 \label{eq:to_ell_2_norm_f}
 \end{equation}
 
 If $\mu(x)$ is an eigenvalue of $B$, solutions $f(x)$ and $g(x)$ are linearly dependent. Thus, there exists a nonzero number $A=A(x)$ such that $g(x)=Af(x)$. 
 By differentiating the Wronskian $W_{n}(f(x),g(x))$ and using equations \eqref{eq:to_ell_2_norm_g} and \eqref{eq:to_ell_2_norm_f}, we obtain
 \begin{align*}
 &W'(f(x),g(x))=f_{n+1}'(x)g_{n}(x)+f_{n+1}(x)g_{n}'(x)-f_{n}'(x)g_{n+1}(x)-f_{n}(x)g_{n+1}'(x)\\
 &=AW_{n}(f'(x),f(x))-A^{-1}W_{n}(g'(x),g(x))
 =\frac{1-x^{2}}{x^{2}}\left(A\sum_{j=n+1}^{\infty}f_{j}^{2}(x)+A^{-1}\sum_{j=-\infty}^{n}g_{j}^{2}(x)\right)\!\!.
 \end{align*}
 Finally, by sending $n\rightarrow-\infty$, we arrive at the formula
 \begin{equation}
 \sum_{j=-\infty}^{\infty}f_{j}^{2}(x)=\frac{x^{2}}{1-x^{2}}A^{-1}W'(f(x),g(x)),  \quad \mbox{ for } 0<|x|<1.
 \label{eq:norm_general_ident}\end{equation}
 The subscript $n$ in the Wronskian has been omitted for it does not dependent on $n$.
 
 In the second part of the proof, we evaluate the quantities $A$ and $W'(f(x),g(x))$ in \eqref{eq:norm_general_ident}.
 Let $x=\alpha^{-1}q^{m}$, with $m>\lfloor\log_{q}|\alpha|\rfloor$ being fixed. Then $\mu(x)$ is an eigenvalue of $B$
 corresponding to eigenvector $\textbf{\textit{v}}_{m}$, where $v_{m,j}=f_{j}(\alpha^{-1}q^{m})$, by Proposition \ref{prop:spec_B}.
 
 First, we find $A$ such that $g(\alpha^{-1}q^{m})=Af(\alpha^{-1}q^{m})$. Since for $k\geq m$, one has
 \begin{align*}
  g_{k}(\alpha^{-1}q^{m})&=\alpha^{k}q^{-mk}\sum_{j=k-m}^{\infty}(q^{m-k+j+1};q)_{\infty}\frac{(-1)^{j}q^{\frac{1}{2}j(j+1)+2mj}}{(q;q)_{j}}\alpha^{-2j}\\
  &=(-1)^{m+k}\alpha^{2m-k}q^{-\frac{1}{2}m(3m+1)+\frac{1}{2}k(k+1)}\sum_{j=0}^{\infty}(q^{j+1};q)_{\infty}\frac{(-1)^{j}q^{\frac{1}{2}j(j+1)+(m+k)j}}{(q;q)_{k-m+j}}\alpha^{-2j},
 \end{align*}
 one gets
 \[
   g_{k}(\alpha^{-1}q^{m})=(-1)^{m+k}\alpha^{2m-k}q^{-\frac{1}{2}m(3m+1)+\frac{1}{2}k(k+1)}\left(1+o\left(1\right)\right)\!, \quad k\to\infty.
 \]
 By taking into account asymptotic formula \eqref{eq:f_n_asympt_+inf} for $f_{k}(\alpha^{-1}q^{m})$, one obtains
 $$A=\lim_{k\rightarrow\infty}\frac{g_{k}(\alpha^{-1}q^{m})}{f_{k}(\alpha^{-1}q^{m})}=(-1)^{m}\alpha^{2m}q^{-\frac{1}{2}m(3m+1)}.$$
 
 Second, by using (\ref{eq:Wronsk_f_g}), one readily verifies that
 $$W'(f_{k}(\alpha^{-1}q^{m}),g_{k}(\alpha^{-1}q^{m}))=(-1)^{m}\alpha^{2}q^{-\frac{1}{2}m(m+3)}(q;q)^{2}_{\infty}.$$
 Identity (\ref{eq:norm_eigenvec_v_n}) then follows from the general formula (\ref{eq:norm_general_ident}).
\end{proof}

\subsection{The spectral resolution of $B$}

The aim of this subsection is to provide an explicit formula for an arbitrary matrix element of the spectral measure
of $B$, i.e., the measure
\[
 E_{k,l}(\cdot):=\langle e_{k},E_{B}(\cdot)e_{l} \rangle, \quad k,l\in\Z,
\]
where $E_{B}$ stands for the projection-valued spectral measure of the self-adjoint operator $B$.

From Proposition \ref{prop:spec_B}, one deduces that the measure $E_{k,l}$ restricted to the Borel sets located outside the interval $[-2,2]$ is discrete
supported on the eigenvalues $\alpha^{-1}q^{m}+\alpha q^{-m}$, $m>\lfloor\log_{q}|\alpha|\rfloor$. Taking also into account Proposition \ref{prop:norm_eigenvec_B}, one gets
\begin{equation}
 E_{k,l}\left(\{\alpha^{-1}q^{m}+\alpha q^{-m}\}\right)=\frac{v_{m,k}v_{m,l}}{\|\textbf{\textit{v}}_{m}\|^{2}}
 =\left(1-\alpha^{-2}q^{2m}\right)\frac{\alpha^{2m}q^{-m(m+1)}}{(q;q)_{\infty}^{2}}f_{k}\left(\alpha^{-1}q^{m}\right)f_{l}\left(\alpha^{-1}q^{m}\right)\!,
\label{eq:spec_meas_eigenval}
\end{equation}
for $k,l\in\Z$ and $\alpha\neq0$.

It remains to determine the part of the spectral measure within the essential spectrum of~$B$. The approach is based
on an expression for the Green function  
\[
 G_{k,l}(z):=\langle e_{k},(B-z)^{-1}e_{l}\rangle, \quad z\notin\spec(B),
\]
which, together with the Stieltjes-Perron inversion formula, allow to express the spectral measure from the Green function as
\begin{equation}
 E_{k,l}((a,b))=\lim_{\delta\to0+}\lim_{\epsilon\to0+}\frac{1}{2\pi\ii}\int_{a+\delta}^{b-\delta}
 \left(G_{k,l}(x+\ii\epsilon)-G_{k,l}(x-\ii\epsilon)\right)\dd x,
\label{eq:SP_inv_formula}
\end{equation}
for any open interval $(a,b)\subset\R$, see \cite[Thm.~XII.2.10]{DunfordSchwartzII}.

First, we will need the following connection formula.

\begin{lem}\label{lem:connect_formula}
 For all $z\notin q^{\Z/2}\cup\{0\}$, $\alpha\in\C\setminus\{0\}$ and $n\in\Z$ one has
 \[
  f_{n}(z)=\frac{\theta_{q}\left(\alpha z^{-1}\right)}{\theta_{q}\left(z^{-2}\right)}g_{n}(z)
  +\frac{\theta_{q}\left(\alpha z\right)}{\theta_{q}\left(z^{2}\right)}g_{n}\left(z^{-1}\right)
 \]
 where $f_{n}$ and $g_{n}$ are given by \eqref{eq:sol_f} and \eqref{eq:sol_g}.
\end{lem}

\begin{proof}
 Let us denote $\tilde{g}_{n}(z):=g_{n}\left(z^{-1}\right)$ as in Remark \ref{rem:symmet_sol}. With the aid of the asymptotic formula \eqref{eq:g_n_asympt_-inf},
 one verifies 
 \[
 W(g,\tilde{g})=z^{-1}\theta_{q}\left(z^{2}\right)\!.
 \]
 Thus, for $z\notin q^{\Z/2}\cup\{0\}$, $g(z)$ and $\tilde{g}(z)$ are two linearly independent solutions of 
 \eqref{eq:eigen_eq_B}. Consequently, there exist coefficients $A(z)$ and $B(z)$ such that
 \begin{equation}
  f_{n}(z)=A(z)g_{n}(z)+B(z)\tilde{g}_{n}(z), \quad \forall n\in\Z.
 \label{eq:lin_komb_AB}
 \end{equation}
 Further, since $f_{n}(z)=f_{n}\left(z^{-1}\right)$, see Remark \ref{rem:symmet_sol}, from the identity \eqref{eq:lin_komb_AB}, it follows 
 \[
  \left(A(z)-B\left(z^{-1}\right)\right)g_{n}(z)=\left(A\left(z^{-1}\right)-B(z)\right)\tilde{g}_{n}(z), \quad \forall n\in\Z.
 \]
 Thus, $B(z)=A\left(z^{-1}\right)$, by the linear independence of $g(z)$ and $\tilde{g}(z)$. Finally, by assuming $|z|<1$
 in \eqref{eq:lin_komb_AB}, sending $n\to-\infty$ and using formulas \eqref{eq:f_n_asympt_-inf_|x|<1} and \eqref{eq:g_n_asympt_-inf},
 one obtains
 \[
  A(z)=\frac{\theta_{q}\left(\alpha z^{-1}\right)}{\theta_{q}\left(z^{-2}\right)}.
 \]
 Hence, the identity from the statement is established for $|z|<1$. Since the formula is invariant under exchange $z\leftrightarrow z^{-1}$, it
 remains true also for $|z|>1$. Finally, the validity of the identity on the unit circle $|z|=1$ follows from the continuity in $z$.
\end{proof}

Assume $0<|z|<1$. Recall that $f_{n}(z)$, given in \eqref{eq:sol_f}, is the solution of \eqref{eq:eigen_eq_B}
which is square summable at $+\infty$, as it follows from \eqref{eq:f_n_asympt_+inf}. Similarly,
$g_{n}(z)$, given in \eqref{eq:sol_g}, is another solution of \eqref{eq:eigen_eq_B} which is square summable 
at $-\infty$, as it follows from \eqref{eq:g_n_asympt_-inf}. Consequently, the operator $G\left(\mu(z)\right)$ 
whose matrix elements are defined by
\begin{equation}
 G_{k,l}\left(\mu(z)\right)=\frac{1}{W(f,g)}\begin{cases}
                                               g_{k}(z)f_{l}(z),& \; k\leq l,\\
                                               g_{l}(z)f_{k}(z),& \; k\geq l,
                                              \end{cases}
\label{eq:G_kl}
\end{equation}
for $0<|z|<1$ and $z\neq\alpha^{-1}q^{m}$, $m>\lfloor\log_{q}|\alpha|\rfloor$, coincides with the resolvent operator $\left(B-\mu(z)\right)^{-1}$.

\begin{prop}
 Let $\alpha\neq0$ and $-2\leq a<b\leq 2$. Then for any $k,l\in\Z$, it holds
 \begin{equation}
  E_{k,l}\left([a,b]\right)=\frac{1}{2\pi}\int_{\phi_{b}}^{\phi_{a}}f_{l}\left(e^{\ii\phi}\right)f_{k}\left(e^{\ii\phi}\right)
  \left|\frac{\left(e^{2\ii\phi};q\right)_{\infty}}{\left(\alpha e^{\ii\phi},q\alpha^{-1}e^{-\ii\phi};q\right)_{\infty}}\right|^{2}\!\dd\phi
 \label{eq:spectr_meas_in_ess}
 \end{equation}
 where $\phi_{a}=\arccos\left(a/2\right)$ and $\phi_{b}=\arccos\left(b/2\right)$. Consequently,
 $\spec_{ac}(B)=[-2,2]$.
\end{prop}

\begin{proof}
 Since $E_{k,l}(\cdot)=E_{l,k}(\cdot)$, we may assume $k\leq l$ throughout the proof.
 The derivation is based on the calculation of the integrand in the Stieltjes-Perron formula \eqref{eq:SP_inv_formula}. First, we investigate
 the limit
 \begin{equation}
  \lim_{\epsilon\to0+}\left(G_{k,l}\left(\mu(z)+\ii\epsilon\right)-G_{k,l}\left(\mu(z)-\ii\epsilon\right)\right).
 \label{eq:lim_Green_integr}
 \end{equation}
 for $\mu(z)\in(-2,2)$.
 
 Note that, if $\mu(z)$ approaches $2\cos\phi$, for $\phi\in(0,\pi)$, from the lower half-plane $\{z\in\C\mid\Im z<0\}$,
 then $z$ approaches the point $e^{\ii\phi}$ from the interior of the unit disk. Consequently, one verifies that
 \[
  \lim_{\substack{\mu(z)\to2\cos\phi \\ \Im\mu(z)<0}}f_{n}(z)=f_{n}\left(e^{\ii\phi}\right)
 \]
 while, by taking complex conjugates, that
 \[
  \lim_{\substack{\mu(z)\to2\cos\phi \\ \Im\mu(z)>0}}f_{n}(z)=f_{n}\left(e^{-\ii\phi}\right)\!.
 \]
 Analogous formulas hold for the function $g_{n}(z)$. By Remark \ref{rem:symmet_sol}, one has 
 $f_{n}\left(e^{\ii\phi}\right)=f_{n}\left(e^{-\ii\phi}\right)$. Note, however, that this is not 
 true in the case of the second solution $g_{n}(x)$.
 
 Now, we can calculate the limit \eqref{eq:lim_Green_integr}. By using \eqref{eq:G_kl}, \eqref{eq:Wronsk_f_g} together with the above limit relations, one gets
 \begin{align*}
  &\lim_{\epsilon\to0+}\!\left(G_{k,l}\left(2\cos\phi+\ii\epsilon\right)-G_{k,l}\left(2\cos\phi-\ii\epsilon\right)\right) 
  =f_{l}\left(e^{\ii\phi}\right)\!\!\left[-\frac{g_{k}\left(e^{-\ii\phi}\right)}{e^{\ii\phi}\theta_{q}\left(\alpha e^{-\ii\phi}\right)}
  +\frac{g_{k}\left(e^{\ii\phi}\right)}{e^{-\ii\phi}\theta_{q}\left(\alpha e^{\ii\phi}\right)}\right]\\
  &=\frac{f_{l}\left(e^{\ii\phi}\right)}{\left|\theta_{q}\left(\alpha e^{\ii\phi}\right)\right|^{2}}
  \left[-e^{-\ii\phi}\theta_{q}\left(\alpha e^{\ii\phi}\right)g_{k}\left(e^{-\ii\phi}\right)+e^{\ii\phi}\theta_{q}\left(\alpha e^{-\ii\phi}\right)g_{k}\left(e^{\ii\phi}\right)\right]
 \end{align*}
 By applying Lemma \ref{lem:connect_formula} together with \eqref{eq:theta_recip_arg}, the above expression in the square brackets can be further simplified
 and one arrives at the relation
 \begin{equation}
  \lim_{\epsilon\to0+}\!\left(G_{k,l}\left(2\cos\phi+\ii\epsilon\right)-G_{k,l}\left(2\cos\phi-\ii\epsilon\right)\right)
  =-\frac{e^{-\ii\phi}\theta_{q}\left(e^{2\ii\phi}\right)}{\left|\theta_{q}\left(\alpha e^{\ii\phi}\right)\right|^{2}}f_{l}\left(e^{\ii\phi}\right)f_{k}\left(e^{\ii\phi}\right)\!,
 \label{eq:lim_eps_dif_G}
 \end{equation}
 for $\phi\in(0,\pi)$.
 
 Hence for any $-2<a<b<2$ we may use the Lebesgue dominating convergence theorem and interchange the limit $\epsilon\to0+$ and the integral in the formula \eqref{eq:SP_inv_formula} which together
 with a change of variables and \eqref{eq:lim_eps_dif_G} yields
 \begin{align*}
 E_{k,l}((a,b))&=\frac{1}{\pi}\int_{\phi_{b}}^{\phi_{a}}f_{l}\left(e^{\ii\phi}\right)f_{k}\left(e^{\ii\phi}\right)\frac{\ii e^{-\ii\phi}\theta_{q}\left(e^{2\ii\phi}\right)}{\left|\theta_{q}\left(\alpha e^{\ii\phi}\right)\right|^{2}}\sin\phi\ \dd\phi\\
 &=\frac{1}{2\pi}\int_{\phi_{b}}^{\phi_{a}}f_{l}\left(e^{\ii\phi}\right)f_{k}\left(e^{\ii\phi}\right)\left|\frac{\left(e^{2\ii\phi};q\right)_{\infty}}{\left(\alpha e^{\ii\phi},q\alpha^{-1}e^{-\ii\phi};q\right)_{\infty}}\right|^{2}\dd\phi.
 \end{align*}
 Finally, the statement follows by verifying that the end points $\pm2$ are not eigenvalues of $B$. This, however, has been shown in Proposition \ref{prop:spec_B}.
\end{proof}

By putting together the discrete part of the spectral measure \eqref{eq:spec_meas_eigenval} with the absolutely continuous part \eqref{eq:spectr_meas_in_ess},
one arrives at the following statement.

\begin{thm}\label{thm:spec_meas_B}
 Let $\alpha\neq0$ and $\mathcal{A}\subset\R$ be a Borel set. Then for all $k,l\in\Z$, it holds
 \begin{align*}
  E_{k,l}(\mathcal{A})&=\frac{1}{2\pi}\int_{\mathcal{A}_{c}}f_{l}\left(e^{\ii\phi}\right)f_{k}\left(e^{\ii\phi}\right)
  \left|\frac{\left(e^{2\ii\phi};q\right)_{\infty}}{\left(\alpha e^{\ii\phi},q\alpha^{-1}e^{-\ii\phi};q\right)_{\infty}}\right|^{2}\!\dd\phi\\
  &+\frac{1}{(q;q)_{\infty}^{2}}\sum_{m\in\mathcal{A}_{d}}\left(1-\alpha^{-2}q^{2m}\right)\alpha^{2m}q^{-m(m+1)}f_{l}\left(\alpha^{-1}q^{m}\right)f_{k}\left(\alpha^{-1}q^{m}\right)\!.
 \end{align*}
 where
 \[
  \mathcal{A}_{c}=\{\phi\in[0,\pi] \mid 2\cos\phi\in[-2,2]\cap\mathcal{A}\} \quad\mbox{ and }\quad
  \mathcal{A}_{d}=\{m>\lfloor\log|\alpha|\rfloor \mid \alpha^{-1}q^{m}+\alpha q^{-m}\in\mathcal{A}\}.
 \]
\end{thm}

\subsection{Consequences for the third Jackson $q$-Bessel function}

 Recall the third Jackson $q$-Bessel function is defined as
 \[
  J_{\nu}(z;q):=\frac{z^{\nu}}{(q;q)_{\infty}}{}_{1}\tilde{\phi}_{1}\left(0;q^{\nu+1};q,qz^{2}\right)\!,
 \]
 see, for example, \cite{KoornwinderSwarttouw, Swarttouw}. In addition, we have the relation \cite[Eq.~3.2.12]{Swarttouw}
 \[
  J_{n}(z;q)=(-1)^{n}q^{-n/2}J_{-n}\left(zq^{-n/2};q\right)\!, \quad n\in\Z.
 \]
 These two formulas allow us to express the elements of the eigenvectors $\textbf{\textit{v}}_{m}$, given by \eqref{eq:eigenvec_v_mj}, in the form
 \begin{equation}
  v_{m,j}=(-1)^{m}\alpha^{-m}q^{m(m+1)/2}(q;q)_{\infty}\,J_{m-j}\left(\alpha^{-1}q^{m};q\right)\!,
  \label{eq:v_mj_rel_qBessel}
 \end{equation}
 for $j\in\Z$ and $m>\lfloor\log_{q}|\alpha|\rfloor$. 

Since $\langle \textbf{\textit{v}}_{n},\textbf{\textit{v}}_{m}\rangle=\|\textbf{\textit{v}}_{m}\|^{2}\delta_{m,n}$,
for $m,n>\lfloor\log_{q}|\alpha|\rfloor$, identities \eqref{eq:v_mj_rel_qBessel} and \eqref{eq:norm_eigenvec_v_n} 
yield the orthogonality relation
\begin{equation}
 \sum_{j\in\Z}J_{j+m}\left(\alpha^{-1}q^{m};q\right)J_{j+n}\left(\alpha^{-1}q^{n};q\right)=\frac{\delta_{m,n}}{1-\alpha^{-2}q^{2m}},
\label{eq:OG_rel_qBess}
\end{equation}
for $m,n>\lfloor\log_{q}|\alpha|\rfloor$ and $\alpha\neq0$.
In particular, if $m=n$, one has
\[
 \sum_{j\in\Z}J_{j}^{2}\left(\alpha^{-1}q^{m};q\right)=\frac{1}{1-\alpha^{-2}q^{2m}}.
\]
Since $\{\alpha^{-1}q^{m} \mid \alpha\neq0 \wedge m>\lfloor\log_{q}|\alpha|\rfloor\}=(-1,1)\setminus\{0\}$,
we arrive at the summation formula
\begin{equation}
 \sum_{j\in\Z}J_{j}^{2}\left(x;q\right)=\frac{1}{1-x^{2}}
\label{eq:sum_form_qBess}
\end{equation}
where $x\in(-1,1)\setminus\{0\}$. However, the validity of \eqref{eq:sum_form_qBess} 
can be extended to $|x|<1$ since the sum on the LHS of \eqref{eq:sum_form_qBess} converges locally uniformly on $|x|<1$, and
the analyticity argument applies.

The summation formula \eqref{eq:sum_form_qBess} as well as the orthogonality relation \eqref{eq:OG_rel_qBess} follow from a more general identity derived by 
Koornwinder and Swarttouw in \cite{KoornwinderSwarttouw}. Namely, one obtains \eqref{eq:sum_form_qBess} by putting $m=n=0$, $s=q^{-1}$ and $z=q^{1/2}x$ in 
\cite[Eq.~(4.8)]{KoornwinderSwarttouw}, while \eqref{eq:OG_rel_qBess} follows from \cite[Eq.~(4.7)]{KoornwinderSwarttouw} by setting $s=q^{-1}$, $x=q^{m+1/2}\alpha^{-1}$
and $y=q^{n+1/2}\alpha^{-1}$. 

Let us point out that \eqref{eq:sum_form_qBess} is a $q$-analogue to the the well known identity for Bessel functions of the first kind  \cite[Eq.~9.1.76]{AbramowitzStegun}
\begin{equation}
 \sum_{j\in\Z}J_{j}^{2}(z)=1,
\label{eq:sum_form_Bess}
\end{equation} 
since by putting $x=(1-q)z$ in \eqref{eq:sum_form_qBess} and sending $q\to1-$, one arrives at \eqref{eq:sum_form_Bess}. Let us also remark that 
Koornwinder and Swarttouw found another $q$-analogue to \eqref{eq:sum_form_Bess}, namely
\[
  \sum_{j\in\Z}q^{j}J_{j}^{2}\left(z;q\right)=1, \quad |z|<1,
\]
as it follows from \cite[Eq.~(2.10)]{KoornwinderSwarttouw} by putting $m=n=0$ and writing $q^{1/2}z$ instead of $z$.

Finally, if we put $\mathcal{A}=\R$ in Theorem \ref{thm:spec_meas_B} we obtain the orthogonality relation for the confluent
basic hypergeometric series as in \eqref{eq:sol_f},
\begin{align*}
  \frac{1}{2\pi}\int_{0}^{\pi}&f_{l}\left(e^{\ii\phi}\right)f_{k}\left(e^{\ii\phi}\right)
  \left|\frac{\left(e^{2\ii\phi};q\right)_{\infty}}{\left(\alpha e^{\ii\phi},q\alpha^{-1}e^{-\ii\phi};q\right)_{\infty}}\right|^{2}\!\dd\phi\\
  &+\frac{1}{(q;q)_{\infty}^{2}}\sum_{m>\lfloor\log|\alpha|\rfloor}\left(1-\alpha^{-2}q^{2m}\right)\alpha^{2m}q^{-m(m+1)}f_{l}\left(\alpha^{-1}q^{m}\right)f_{k}\left(\alpha^{-1}q^{m}\right)=\delta_{l,k},
\end{align*}
for all $k,l\in\Z$ and $\alpha\neq0$. Note that the absolutely continuous part of the measure in the above integral coincides with the Askey--Wilson measure with a
particular choice of parameters, namely $a=\alpha$, $b=\alpha^{-1}q$ and $c=d=0$, see \cite[Chp.~6]{GasperRahman} and \cite{AskeyWilson}.

\section*{Acknowledgements}
Research of Mourad E. H. Ismail was partially supported by the DSFP of King Saud 
 University and by  the National Plan for Science, Technology and innovation (MAARIFAH), 
 King Abdelaziz City for Science and Technology, Kingdom of Saudi Arabia, Award number 14-MAT623-02.
 F.~{\v S}tampach thanks to Professor Tom Koornwinder for pointing out the connection between 
 results of \cite{KoornwinderSwarttouw} and formulas \eqref{eq:OG_rel_qBess} and \eqref{eq:sum_form_qBess}.


\begin{thebibliography}{8}

\bibitem{AbramowitzStegun} M.~Abramowitz, I.~A.~Stegun: \emph{Handbook of mathematical functions with formulas, graphs,
and mathematical tables}, (U.S. Government Printing Office, Washington, D.C., 1964).

\bibitem{Akhiezer} N. I. Akhiezer: \emph{The Classical Moment Problem and Some Related Questions in Analysis}, 
(Oliver and Boyed, Edinburgh, 1965).

\bibitem{AlSalamIsmail} W.~A.~Al-Salam, M.~E.~H.~Ismail: \emph{Orthogonal polynomials associated with the Rogers-Ramanujan continued fraction}, 
Pac. J.~Math. \textbf{104} (1983) 269--283.
 
\bibitem{AndrewsAskeyRoy}  G.~E. Andrews, R.~Askey, R.~Roy: \emph{Special Functions}, (Cambridge University Press, Cambridge, 1999).

\bibitem{AskeyWilson} R.~Askey, J.~Wilson: \emph{Some basic hypergeometric orthogonal polynomials that generalize Jacobi polynomials},
Mem. Amer. Math. Soc. \textbf{54} (1985) iv+55 pp. 

\bibitem{Berezanskii} Ju.~M.~Berezans'ki{\u\i}: \emph{Expansions in eigenfunctions of selfadjoint operators}, 
(Translations of Mathematical Monographs, American Mathematical Society, Providence, Rhode Island, 1968).

\bibitem{Carlitz} L. Carlitz: \emph{Some formulas related to the Rogers--Ramanujan identities}, Annali di Math. (IV)  {\bf 47} (1959) 243--251.

\bibitem{ChenIsmail} Y.~Chen, M.~E.~H.~Ismail: \emph{Some indeterminate moment problems and Freud-like weights},
Constr. Approx. \textbf{14} (1998) 439--458.

\bibitem{DunfordSchwartzII} N.~Dunford, J.~T.~Schwartz: \emph{Linear operators. Part II: Spectral theory. Self adjoint operators in Hilbert space}, 
(Interscience Publishers John Wiley \& Sons, New York-London, 1963).

\bibitem{GasperRahman} G. Gasper, M. Rahman: \emph{Basic Hypergeometric Series}, (second edition, Cambridge University Press, Cambridge, 2004). 

\bibitem{GarrretIsmailStanton} K. Garrett, M. E. H. Ismail, D. Stanton: \emph{Variants of the Rogers--Ramanujan identities},
Advances in Applied Math. {\bf 23} (1999), 274--299.

\bibitem{Ismail05} M. E. H. Ismail: \emph{Asymptotics of $q$-orthogonal polynomials and a $q$-Airy function}, 
Int. Math. Res. Not. \textbf{18} (2005) 1063--1088.

\bibitem{Ismail} M. E. H. Ismail:  \emph{Classical and Quantum Orthogonal Polynomials in one Variable}, 
(paperback edition, Cambridge University Press, Cambridge, 2009).

\bibitem{IsmailMulla} M.~E.~H.~Ismail, F.~S.~Mulla: \emph{On the generalized Chebyshev polynomials}, SIAM J. Math. Anal.
\textbf{18} (1987) 243--258.

\bibitem{KoornwinderSwarttouw} T.~H.~Koornwinder, R.~F.~Swarttouw: \emph{On $q$-analogues of the Fourier and Hankel transforms}, 
Trans. Amer. Math. Soc.\textbf{333} (1992) 445--461.

\bibitem{MassonRepka} D.~R.~Masson, J.~Repka: \emph{Spectral theory of Jacobi matrices $\ell^{2}(\Z)$ and the $\mathfrak{su}(1,1)$
Lie algebra}, SIAM J. Math. Anal. \textbf{22} (1991) 1131--1146.

\bibitem{Lawden} D.~F.~Lawden: \emph{Elliptic Functions and Applications}, (Springer-Verlag, New York, 1989).

\bibitem{Olver} F.~W.~J.~Olver: \emph{Asymptotics and special functions}, (A.~K.~Peters Ltd., Wellesley, 1997).

\bibitem{Ramanujan} S.~Ramanujan: \emph{The Lost Notebook and Other Unpublished Papers} (with an Introduction by G.E. Andrews), Narosa, New Delhi (1988).

\bibitem{Schur} I.~Schur: \emph{Ein Beitrag zur Additiven Zahlentheorie und zur Theorie Kettenbr\"{u}che},
S.-B. Preuss, Akad. Wiss. Phys.-Math. Kl. (1917) 302--321.

\bibitem{Stampach} F.~\v{S}tampach: \emph{Nevanlinna extremal measures for polynomials related to $q^{-1}$-Fibonacci polynomials},
Adv. Appl. Math. \textbf{78} (2016) 56--75.

\bibitem{Swarttouw} R.~F.~Swarttouw: \emph{On Hahn-Exton $q$-Bessel function}, (Ph.~D. thesis, Delft Technical University, 1992).

\bibitem{Teschl} G.~Teschl: \emph{Jacobi operators and completely integrable nonlinear lattices}, (AMS, Rhode Island, 2000).

\bibitem{Weidmann} J.~Weidmann: \emph{Linear operators in Hilbert spaces}, (Springer-Verlag, New York-Berlin, 1980).

\end{thebibliography}
\end{document}